\documentclass[9pt,twoside]{amsart}
\usepackage{amsxtra}
\usepackage{color}
\usepackage{amsmath,amsthm,amssymb,bm}
\usepackage{mathrsfs,mathtools}
\usepackage{stmaryrd}
\usepackage{hyperref}
\usepackage{enumerate}
\usepackage{xcolor}
\usepackage{pifont}
\usepackage{adjustbox}
\usepackage{tikz}
\usepackage{comment}

\newtheorem{theorem}{Theorem}[section]

\newtheorem{proposition}[theorem]{Proposition}
\newtheorem{lemma}[theorem]{Lemma}
\newtheorem*{theorem*}{Theorem}

\newtheorem{remark}[theorem]{Remark}

\newcommand{\R}{{\mathbb R}}

\newcommand{\C}{{\mathbb C}}

\newcommand{\T}{{\mathbb T}}

\newcommand{\beq}{\begin{equation}}
\newcommand{\eeq}{\end{equation}}
\renewcommand{\a}{\alpha}
\renewcommand{\b}{\beta}

\renewcommand{\d}{\delta}

\newcommand{\g}{\gamma}

\renewcommand{\l}{\lambda}

\renewcommand{\o}{\omega}

\newcommand{\SU}{{\mathrm{SU}}}

\newcommand{\SO}{{\mathrm {SO}}}

\renewcommand{\T}{{\mathrm T}}

\newcommand{\n}{\nabla}

\DeclareMathOperator\tr{tr\;}

\DeclareMathOperator\End{End}

\DeclareMathOperator\Ad{Ad}
\DeclareMathOperator\ad{ad}

\DeclareMathOperator{\Span}{Span}

\newcommand{\ch}{{\rm{cosh}}}
\newcommand{\sh}{{\rm{sinh}}}

\renewcommand{\ggg}{\mathfrak{g}}
\newcommand{\gh}{\mathfrak{h}}

\newcommand{\gm}{\mathfrak{m}}
\newcommand{\gn}{\mathfrak{n}}

\newcommand{\gp}{\mathfrak{p}}
\newcommand{\gq}{\mathfrak{q}}

\newcommand{\so}{\mathfrak{so}}
\newcommand{\su}{\mathfrak{su}}

\newcommand{\diag}{{\rm diag}}

\numberwithin{equation}{section}

\title[Complete pluriclosed metrics with vanishing Bismut Ricci form on the affine quadric]{Complete pluriclosed metrics with vanishing Bismut Ricci form on the affine quadric}

\author{Fabio Podest\`a and Alberto Raffero}
\address{Dipartimento di Matematica e Informatica ``U.~Dini'' \\ Universit\`a degli Studi di Firenze\\ Viale Morgagni 67/a\\ 50134 Firenze\\ Italy}
\email{fabio.podesta@unifi.it}
\address{Dipartimento di Matematica ``G.~Peano'' \\ Universit\`a degli Studi di Torino\\
Via Carlo Alberto 10\\
10123 Torino\\ Italy}
\email{alberto.raffero@unito.it}


\begin{document}

\begin{abstract}
We study a one-parameter family of $\mathrm{SO}(4)$-invariant pluriclosed Hermitian metrics $g_c$ with vanishing Bismut Ricci form
on the affine quadric $Q_3\cong TS^3$, previously described in the physics literature. 
For every real parameter $c$, we prove local existence and uniqueness of a real-analytic solution to the defining singular initial-value problem,
together with positivity of the associated metric near the singular orbit $S^3$. We then establish global existence and
completeness for $|c|\leq1$ and finite-time degeneration for $|c|>1$. 
The complete family exhausts all  $\SO(4)$-invariant Bismut Hermitian-Einstein metrics on the affine quadric $Q_3$, 
and joins Stenzel's K\"ahler Ricci-flat to a metric that was first considered by Chamseddine--Volkov/Maldacena--Nu\~nez.
We also determine the leading asymptotics and derive an explicit scalar curvature formula, proving strict positivity for
the non-K\"ahler members and identifying the change in asymptotic scalar curvature at the endpoint $c=1$. 
Finally, we show that all these metrics are not Bismut flat and have full Bismut holonomy $\SU(3)$, providing new examples of such manifolds.
\end{abstract}

\maketitle

\section{Introduction}
Hermitian metrics with closed Bismut torsion and vanishing Bismut Ricci form provide a natural extension of K\"ahler Ricci-flat geometry. 
These conditions combine the pluriclosed equation $dd^c\omega=0$ with the vanishing of the Bismut-Ricci form $\rho^B$. 
According to the recent mathematical literature, we refer to metrics satisfying both conditions as
{\em Bismut Hermitian--Einstein} metrics (shortly BHE metrics), see \cite{GFS,GJS,ABLS,ALL,ALLR,BFG,BFGV}. 

In this paper, we study such metrics on the affine quadric
\[
Q_3=\left\{
z\in\mathbb{C}^4:\sum_{j=1}^4 z_j^2=1
\right\},
\]
equipped with its standard complex structure and the cohomogeneity-one action of $\mathrm{SO}(4)$. 
The manifold $Q_3$ is diffeomorphic to $TS^3$, with the zero section forming the singular orbit. 

The Bismut Hermitian-Einstein condition for an $\mathrm{SO}(4)$-invariant Hermitian metric on $Q_3$ reduces to a singular ordinary
differential equation for a function $v_c$, depending on a real parameter $c$.
The corresponding family of Bismut Hermitian-Einstein metrics $g_c$ can be traced back in the supergravity construction of Butti, Gra\~na, Minasian, Petrini and Zaffaroni
\cite{BGMPZ}, describing the baryonic branch of the Klebanov--Strassler solution. 
The relevant torsional backgrounds arise as a limit of that construction, were also studied by Casero, N\'u\~nez and Paredes \cite{CNP}, and were presented
in an explicit unwarped formulation by Maldacena and Martelli \cite{MM}. 
These works derive the relevant differential equations from supersymmetry and investigate their solutions through a combination of series expansions, 
numerical integration, and analytical approximations in asymptotic and limiting regimes. 
Although this family was described in the supergravity setting, it appears to have received little attention in the mathematical study of
Bismut Hermitian--Einstein metrics.

The present work makes this connection explicit by studying $\mathrm{SO}(4)$-invariant Hermitian metrics on the affine quadric $Q_3$. 
Within this invariant setting, we reduce the pluriclosed condition and the vanishing of the Bismut-Ricci form to a singular ordinary
differential equation. 
We then provide a rigorous local and global analysis of the resulting singular initial-value problem, and we obtain the following main result.
\begin{theorem} On $Q_3$ there exists an analytic family of complete $\SO(4)$-invariant metrics $g_c$, $c\in [0,1]$, with the following properties: 
\begin{enumerate}[i)]
\item $g_0$ is homothetic to the Stenzel's K\"ahler Ricci-flat metric; 
\item $g_c$ is pluriclosed for every $c$ and it is non-K\"ahler for $c>0$;
\item $g_c$ has zero Bismut-Ricci form $\rho^B$ and non-zero Bismut curvature;
\item along a normal transverse curve $\gamma_r$, the metric $g_c$ is determined by smooth functions $u_c,v_c,w_c\in C^\infty(\mathbb R)$ with
$$ 
u_c=v_c',\qquad 
w_c(r)=c\left(r\coth(2r)-\tfrac12\right).
$$
For every $c\in[0,1)$ there exist constants $M_1,M_2>0$ such that
$$ 
v_c(r)=M_1 e^{\frac43 r}(1+o(1)),
\qquad
v_c'(r)=M_2 e^{\frac43 r}(1+o(1)),
\qquad r\to+\infty,
$$
whereas, for the critical value $c=1$,
$$
v_1(r)=r,\qquad u_1(r)=1;$$
\item the scalar curvature $\mathrm{Scal}_{g_c}$ is everywhere positive when $c>0$ and along the normal curve 
\begin{align*}
\lim_{r\to+\infty}{\rm{Scal}}_{g_c}(r)&=0;\qquad c\in (0,1),\\
\lim_{r\to+\infty}{\rm{Scal}}_{g_1}(r)&=6;\end{align*}
\item each $g_c$ is non-homogeneous and different $g_c,g_{c'}$ are not isometric;
\item the holonomy algebra of the Bismut connection of $g_c$ is $\mathfrak{su}(3)$ for all $c\in[0,1]$.  
\end{enumerate}
\end{theorem}
Actually, every analytic $\SO(4)$-invariant BHE metric on $Q_3$ is, up to holomorphic isometry and positive scaling, 
represented by a unique member of the family $g_c$, with $c\in[0,1]$.
Moreover, this family connects two distinguished geometries: the Stenzel's K\"ahler Ricci-flat metric~\cite{S} at $c=0$ and 
the Chamseddine--Volkov/Maldacena--Nu\~nez metric \cite{CV1,CV2,MN} at $c=1$. 

The family constructed here fits into a phenomenon that seems to be emerging in the study of complete non-compact generalized Ricci solitons 
and related string backgrounds: genuinely torsionful solutions may arise in continuous families deforming classical Ricci solitons or 
special holonomy metrics. In the present setting, the Stenzel's K\"ahler Ricci-flat metric appears as the zero-torsion member of a family of complete 
non-K\"ahler Bismut Hermitian-Einstein metrics. 
Similar deformation phenomena occur for Bryant-type generalized Ricci solitons, for complete Ricci-Yang-Mills solitons, and, more recently, 
for toric generalized K\"ahler-Ricci solitons (see e.g. \cite{Womack,ABSU,PR}). 
This behavior contrasts with the substantially stronger rigidity phenomena known in the compact setting, where the simultaneous occurrence of classical 
and genuinely torsionful geometries is often severely restricted or excluded (see e.g. \cite{StreetsTopology}).

\smallskip

The strategy of the proof is the following.  
We fix a normal transverse curve $\gamma_r$ and derive the expression of a Hermitian $\mathrm{SO}(4)$-invariant metric on the regular part of $Q_3$. 
If we then impose that the metric is pluriclosed and smoothly extends across the singular orbit $S^3$, 
we see that the metric only depends on a function $v(r)$ of one variable and a real constant $c$ that appears in the function  
\[
w_c(r)=c\left(r\coth(2r)-\tfrac12\right),
\]
see Lemma \ref{cond}. 
The necessary and sufficient conditions on $v$ guaranteeing the extendability of the metric on the whole $Q_3$ are then studied 
and proven to be equivalent to the extendability of the function $v$ as an odd function on the whole $\mathbb R$ 
(Proposition \ref{prop:extension}).\par 
In order to compute the Bismut-Ricci form $\rho^B$, we use the fact that the quadric $Q_3$ can be seen as the complex homogeneous space $Q_3=\SO(4,\mathbb C)/\SO(3,\mathbb C)$ and that there is an $\SO(4,\mathbb C)$-invariant holomorphic $(3,0)$-form $\Omega$. The Bismut Ricci form $\rho^B$ is then written along the curve $\g_r$ and the equation $\rho^B=0$ turns out to be equivalent to a singular second order ODE involving $v$, namely 

\[
v''
+2v'\frac{vv'+w_cw_c'}{v^2-w_c^2}
-4v'\coth(2r)=0,
\qquad
v(0)=0,\quad v'(0)=1,
\]
where the condition $v'(0)=1$ fixes the homothetic class of the metric.

We first prove that this ODE admits a unique real-analytic odd solution $v_c$ in a neighbourhood of the origin and that $v_c$ depends analytically on the parameter c by applying \cite{L} 
(see Proposition \ref{local}). 
Then, we discuss the long time behaviour of the solution in Theorem \ref{thm:completeness}, where we show that $v_c$ exists on the whole $\R$ for $|c|<1$ 
and gives rise to a complete metric $g_c$, while it develops a finite time singularity for $|c|>1$.

The expression of the scalar curvature of $g_c$ is determined in Section \ref{sect:scalar}, where we also compute its value on the singular orbit and discuss 
its behaviour at infinity. This allows us to show that the metrics $g_c$ are pairwise not isometric, for all $c\in[0,1]$, and not homogeneous 
(Proposition \ref{prop:inhomogeneous}). 
By \cite[Theorem 5]{Z}, this last property ensures that they are not Bismut flat. \par
In the last section, we show that the Bismut holonomy algebra of the metrics $g_c$ coincides with the full $\su(3)$ when $c\in [0,1]$. 
To the best of our knowledge, the metrics $g_c$ with $c\in(0,1]$ provide the first explicit complete non-Kähler Bismut Hermitian–Einstein metrics 
for which the Bismut connection is shown to have full \(\SU(3)\)-holonomy. 
This is in stark contrast to the compact case, where a compact non-K\"ahler Bismut Hermitian-Einstein manifold of complex dimension $n$ 
has the Bismut holonomy algebra contained in $\mathfrak{su}(n-1)$, as it immediately follows from \cite[Prop.~2.6]{ABLS}. 

\smallskip 

Finally, we remark that the same construction can not be applied to the higher dimensional case $Q_n\cong TS^n$, 
by a simple representation-theoretical reason. 
Indeed the space of Hermitian $\SO(n+1)$-invariant metrics can be easily described along a normal geodesic by looking at the space of $L$-invariant metrics 
on each principal orbit, where the principal isotropy $L\cong \SO(n-1)$. 
When $n\geq 4$ it is immediately noticeable that the dimension of this space drops by one compared to the case $n=3$ 
and the pluriclosed condition forces the metric to be K\"ahler.\par

\smallskip

\noindent{\bf Notation.} If $X$ is an element of the Lie algebra $\ggg$, then $\hat X$ denotes the induced vector field on the manifold which is acted on by the Lie group $G$.

\section{Preliminaries}
We consider the affine quadric 
$$Q_3=\left\{z\in \mathbb C^4|\ \sum_{i=1}^4 z_i^2 = 1\right\}$$
which can be identified with the tangent bundle to the sphere $S^3$ by means of the map $\Phi:TS^3\to Q_3$ given by 
$$\Phi(x,v) = \ch(|v|)\ x + i \frac{\sh (|v|)}{|v|} v,\qquad (x,v)\in TS^3.$$
The group $G=\SO(4)$ acts on $Q_3$ by cohomogeneity one and we can consider the transversal curve 
$$\g(r)= i\ \sh(r)e_3 + \ch(r) e_4,\qquad r\in (0,+\infty),$$
so that the stabilizer at points with $r>0$ is given by $H=\SO(2)$ standardly embedded into $G$ as the subgroup fixing both $e_3,e_4$.

We denote by $\xi$ the field $\g'_r$ (extended to a $G$-invariant vector field on the whole complement of $S^3$ in $Q_3$) and we fix the standard vectors 
$$A\coloneqq E_{34},\ X_i\coloneqq E_{i3},\ Y_i\coloneqq E_{i4},\qquad i=1,2,$$
where $E_{ij}$ denotes the skew matrix in $\so(4)$ whose entry $(i,j)$ is given by $1$. 
We also set from now on 
$$\T(r)\coloneqq \tanh(r),$$
and we will frequently omit the variable $r$ for brevity. 

Let $J$ denote the complex structure on $Q_3$ obtained by pulling back the multiplication by $\sqrt{-1}$ on $\C^4$.
A simple computation shows that along the curve $\g$ we have 
$$J\xi = -\hat A,\quad J\hat X_i =  -{\rm{T}}(r) \hat Y_i,\quad J\hat Y_i = \frac 1{{\rm{T}}(r)}\hat X_i,\qquad i=1,2.$$

The tangent space $T_{\g_r}Q_3$ can be written as 
$$T_{\g_r}Q_3 = \mathbb R\cdot \xi + \mathbb R\cdot \hat A + \gp_1 + \gp_2,$$
where $\gp_1 = \operatorname{Span}\{X_1,X_2\}$, $\gp_2= \operatorname{Span}\{Y_1,Y_2\}$ are equivalent $\so(2)$-modules. 
The tangent space to the regular orbit $G/H$ is therefore identified with $\gm\coloneqq \mathbb R\cdot \hat A + \gp_1 + \gp_2$, 
We note that any $\so(2)$-invariant scalar product on $\gp_1+\gp_2$ is represented by a matrix of the following form 
with respect to the basis $\{X_1,X_2,Y_1,Y_2\}$
$$h = \left(\begin{matrix} a^2&0&b&c\\ 0&a^2&-c&b\\ b&-c&d^2&0\\ c&b&0&d^2\end{matrix}\right),$$
as $\operatorname{End}_{\so(2)}(\mathbb R^2) = \operatorname{Span}\{ \mathrm{Id}, K\}$, 
where $K = \left(\begin{matrix}0&1\\-1&0\end{matrix}\right)$.
If we now impose the Hermitian condition, we immediately find that 
$$b= h(X_1,Y_1)=h(JX_1,JY_1) = - h(Y_1,X_1) = 0,$$
and 
$$d^2 = h(Y_1,Y_1) = h(JY_1,JY_1) = \frac 1{{\rm{T}}(r)^2}h(X_1,X_1) = \frac 1{{\rm{T}}(r)^2} a^2.$$
The positive definiteness of $h$ implies 
$$a^4> ({\rm{T}}(r) c)^2.$$

Let now $\{dr,\eta,x_1,x_2,y_1,y_2\}$ be the dual coframe of the basis $\{\xi,\hat A,\hat X_i,\hat Y_i\}$ along the curve $\g$. 
We can write the generic expression for a $G$-invariant Hermitian metric $g$ on the regular part of $Q_3$ along the curve $\g$ as follows
\begin{equation}\label{g}
g = u(r)(dr^2+\eta^2) + v(r)\left[{\rm{T}}(r) (x_1^2+x_2^2) + \frac 1{{\rm{T}}(r)}(y_1^2+y_2^2)\right]+ w(r)(x_1\odot y_2 - x_2\odot y_1)\notag,
\end{equation}
where $u,v,w$ are smooth functions on $(0,+\infty)$ with $$u,v>0,\qquad v> |w|,$$ 
and $\phi\odot\psi := \phi\otimes \psi + \psi\otimes\phi$.
The associated $(1,1)$-form $\o=g(J\cdot,\cdot)$ has the following expression
\beq\label{omega} \o = -u(r)\ dr\wedge\eta -v(r) (x_1\wedge y_1 +  x_2\wedge y_2) + w(r) \left(\T(r) x_1\wedge x_2 + \frac{1}{\T(r)} y_1\wedge y_2\right). \eeq

\subsection{Computation of $d\omega$.} We now compute the differential $d\omega$. We will need the following relations, which can be easily obtained
\begin{equation}\label{bracket}
\begin{split}
[A,X_i] &=-Y_i,\quad [A,Y_i]=X_i,\\
[X_i,Y_j] &= -\d_{ij} A,\quad [X_i,X_j]_\gm = [Y_i,Y_j]_\gm = 0,\quad i,j=1,2.
\end{split}
\end{equation}
\noindent $\bullet$\ It is known that for $B,C,D\in\so(4)$ we have (we recall that $[\hat B,\hat C] = -\widehat{[B,C]}$)
$$d\o(\hat B,\hat C,\hat D)= -\o(\widehat{[B,C]},\hat D) -\o(\widehat{[D,B]},\hat C)-\o(\widehat{[C,D]},\hat B). $$
Therefore 
$$d\o(\hat A,\hat X_1,\hat X_2) = \o(Y_1,X_2) -  \o(Y_2,X_1) = {\rm{T}}(r) g(Y_2,Y_1)- {\rm{T}}(r) g(Y_1,Y_2) = 0.$$
In a similar fashion we get $d\o(\hat A,\hat Y_1,\hat Y_2) = 0$. We then have 
\begin{align*} d\o(\hat A,\hat X_1,\hat Y_j) &= \o(\hat Y_1,\hat Y_j) + \o(\hat X_j,\hat X_1) +\d_{ij}\o(\hat A,\hat A) \\
{}&= \left(\frac 1{{\rm{T}}(r)} - {\rm{T}}(r)\right) g(Y_j,X_1) = \d_{j2} \left(\frac 1{{\rm{T}}(r)} - {\rm{T}}(r)\right) w(r),
\end{align*}
and similarly for $d\o(\hat A,\hat X_2,\hat Y_1) = - (\frac 1{{\rm{T}}(r)} - {\rm{T}}(r))w(r)$.

Finally, we easily see that 
$$d\o(\hat X_1,\hat X_2,\hat Y_j) =  d\o(\hat Y_1,\hat Y_2,\hat X_j) =0 , \quad j=1,2,$$
as $J\hat A$ is $g$-orthogonal to $\hat \gp_i$ for $i=1,2$. Therefore the only non-trivial contributions are given by 
\beq\label{eq1} 
d\o(\hat A,\hat X_1,\hat Y_2) = \left(\frac 1{{\rm{T}}(r)} - {\rm{T}}(r)\right) w(r), \qquad 
d\o(\hat A,\hat X_2,\hat Y_1) = - \left(\frac 1{{\rm{T}}(r)} - {\rm{T}}(r)\right)w(r),\eeq
\noindent $\bullet$\ We now consider for $B,C\in \so(4)$
$$d\o(\xi,\hat B,\hat C) = \frac{d}{dr} \o(\hat B,\hat C) + \o(\xi,\widehat{[B,C]}).$$
We first remark that 
$$d\o(\xi,\hat A,\hat B) = 0\quad \forall B\in \so(4),$$
as both $\xi$ and $\hat A$ are invariant under the isotropy representation, that has no invariant vector in $\gp_1+\gp_2$.\par 
We next consider 
\[
d\o(\xi,\hat X_1,\hat X_2) =  \frac{d}{dr}\o(\hat X_1,\hat X_2) = -\frac{d}{dr}({\rm{T}}(r)\ g(\hat Y_1,\hat X_2)) 
= \frac{d}{dr}({\rm{T}}(r)\ w(r)),
\]
$$ d\o(\xi,\hat Y_1,\hat Y_2) =  \frac{d}{dr}g(J\hat Y_1,\hat Y_2) = \frac{d}{dr}\left(\frac{w(r)}{{\rm{T}}(r)}\right).$$
Finally, 
\begin{align*} 
d\o(\xi,X_i,Y_j) &= \frac{d}{dr} g(J\hat X_i,Y_j) - g(\hat A,[X_i,Y_j]) \\
{}&= \d_{ij}\bigl[-\frac{d}{dr}\bigl({\rm{T}}(r)\frac{v(r)}{{\rm{T}}(r)}\bigr) + u(r)\bigr]\\
{}&= (-v'(r)+u(r))\d_{ij}.
\end{align*}
Summing up, we have 
\begin{equation}\label{domega}
    \begin{split}
        d\omega &= \left(\frac {1}{\T}-\T \right)w\, \eta \wedge ( x_1\wedge y_2 - x_2\wedge y_1) \\
                &\quad+ dr\wedge \left((\T w)' x_1\wedge x_2 + \left(\frac w\T\right)' y_1\wedge y_2\right) \\
                &\quad+ (u-v') dr\wedge (x_1\wedge y_1 + x_2 \wedge y_2).
    \end{split} 
\end{equation} 

\begin{remark}
In particular, the K\"ahler condition $d\o=0$ becomes 
$$w\equiv 0,\quad u= v'.$$
\end{remark}

\subsection{Computation of $dd^c\omega$.} 

First of all, we note that, if $B,C,D\in \gp_1+\gp_2$, then along the curve $\g$ we have
$$dd^c\o(\xi, \hat B,\hat C,\hat D) = 0,\qquad dd^c\o(\hat A, \hat B,\hat C,\hat D) = 0$$
as $\xi,\hat A$ are isotropy invariant and $(\gp_1+\gp_2)^\gh =\{0\}$.

We next recall the well-known formula for the differential of an invariant form, namely for $B_0,\ldots,B_3\in\so(4)$
$$dd^c\o(\hat B_0,\ldots,\hat B_3) = \sum_{i<j} (-1)^{i+j}d^c\o(\widehat{[B_i,B_j]},\hat B_{h(i,j)},\hat B_{k(i,j)}),$$
where $h(i,j)<k(i,j)$ and $\{i,j,h(i,j),k(i,j)\}=\{0,1,2,3\}$. 

Using the convention $d^c\omega=-d\omega(J\cdot,J\cdot,J\cdot)$, we now compute 
\begin{align*}
dd^c\o(\hat X_1,\hat X_2,\hat Y_1,\hat Y_2) &= d^c\o(\widehat{[X_1,Y_1]},X_2,Y_2) + d^c\o(\widehat{[X_2,Y_2]},X_1,Y_1) \\ 
{}&= d\o(J\hat A,J\hat X_2,J\hat Y_2) +d\o(J\hat A,J\hat X_1,J\hat Y_1) \\
{}&= -d\o(\xi,\hat Y_2,\hat X_2) -  d\o(\xi,\hat Y_1,\hat X_1) = 2(u-v').
\end{align*}
We now need to compute $dd^c\o(\xi,\hat A, \hat B,\hat C)$, where $B,C\in \gp_1+\gp_2$. We start with the following general formula 
\begin{align*}
dd^c\o(\xi,\hat A,\hat B,\hat C) &= \frac d{dr}d^c\o(\hat A,\hat B,\hat C) \\
{}&\quad+ d^c\o(\xi,\widehat{[A,B]},\hat C) + d^c\o(\xi,\widehat{[C,A]},\hat B) + d^c\o(\xi,\widehat{[B,C]},\hat A).
\end{align*}
It then follows 
\begin{align*} 
dd^c\o(\xi,\hat A,\hat X_1,\hat X_2) &= -\frac d{dr}({\rm{T}}(r)^2\ d\o(\xi,\hat Y_1,\hat Y_2)) \\
{}&\quad - d^c\o(\xi,\hat Y_1,\hat X_2) + d^c\o(\xi,\hat Y_2,\hat X_1)\\
{}&= -\frac d{dr}\left( {\rm{T}}(r)^2 \frac{d}{dr}\left(\frac{w(r)}{{\rm{T}}(r)}\right)\right) + d\o(\hat A,\hat X_1,\hat Y_2) - d\o(\hat A,\hat X_2,\hat Y_1)\\
{}&= -\frac d{dr}\left( {\rm{T}}(r)^2 \frac{d}{dr}\left(\frac{w(r)}{{\rm{T}}(r)}\right)\right) + 2\left(\frac 1{{\rm{T}}(r)}-{\rm{T}}(r)\right)w(r).
\end{align*} 

We now compute 
\begin{align*} dd^c\o(\xi,\hat A,\hat X_1,\hat Y_1) &= \frac d{dr} d^c\o(\hat A,\hat X_1,\hat Y_1) \\
{}&\quad + d^c\o(\xi, -\hat Y_1,\hat Y_1) + d^c\o(\xi,- \hat X_1,\hat X_1) + d^c\o(\xi, -\hat A,\hat A)\\
{}&= (v'-u)'.
\end{align*}
Moreover 
\begin{align*} dd^c\o(\xi,\hat A,\hat X_1,\hat Y_2) &= \frac d{dr} d^c\o(\hat A,\hat X_1,\hat Y_2) \\
{}& \quad+ d^c\o(\xi, -\hat Y_1,Y_2) + d^c\o(\xi,- \hat X_2,X_1) \\
{}&=\frac d{dr} d\o(\xi,\hat Y_1,\hat X_2) - \frac 1{{\rm{T}}(r)^2} d\o(\hat A,\hat X_1,\hat X_2) - {\rm{T}}(r)^2\ d\o(\hat A,\hat Y_1,\hat Y_2) = 0.
\end{align*}
We now recall that the form $dd^c\o$ is of type $(2,2)$ and therefore it is $J$-invariant. Therefore 
$$dd^c\o(\xi,\hat A,\hat X_1,\hat X_2)= {\rm{T}}(r)^2\ dd^c\o(-\hat A,\xi,-\hat Y_1,\hat Y_2) = {\rm{T}}(r)^2\ dd^c\o(\xi,\hat A,\hat Y_1,\hat Y_2)$$
and similarly
$$dd^c\o(\xi,\hat A,\hat X_1,\hat Y_2) =  dd^c\o(\xi,\hat A,\hat X_2,\hat Y_1)=0.$$
Finally, the isotropy action implies that 
$$dd^c\o(\xi,\hat A,\hat X_1,\hat Y_1) = dd^c\o(\xi,\hat A,\hat X_2,\hat Y_2).$$
Therefore we have
\begin{lemma}\label{cond} The metric $g$ is pluriclosed if and only if the following conditions hold 
\begin{enumerate}[$(1)$]
\item $v'(r)=u(r)$; \vspace{0.1cm}
\item\label{2} $\frac d{dr}\bigl( {\rm{T}}(r)^2 \frac{d}{dr}\bigl(\frac{w(r)}{{\rm{T}}(r)}\bigr)\bigr) - 2\bigl(\frac 1{{\rm{T}}(r)}-{\rm{T}}(r)\bigr)w(r)=0.$
\end{enumerate}
\end{lemma}
Equation \eqref{2} can be explicitely integrated to yield the solutions 
$$w(r) = c_1 \coth (2r) + c_2\left(r\coth(2r)-\frac 12\right),$$
with $c_1,c_2\in \mathbb R$. Now the function 
$$w(r) = g(\hat X_1,\hat Y_2)|_{\g(r)}$$
and therefore, if we like to have $g$ extendable around the singular orbit $S^3$, we must have $\lim_{r\to 0} w(r)=0$. As $\lim_{r\to 0}r\coth(2r) = \frac 12$, we must have $c_1=0$, hence the admissible function $w(r)$ has the form 
\beq\label{w}
w(r) = c \left(r\coth(2r)-\frac 12\right),\qquad c\in\mathbb R.
\eeq

\subsection{Smoothness conditions for the extendability of the metric across the singular orbit $S^3$}
We consider the singular point $q=e_4\in Q_3$, whose isotropy is the subgroup $L\cong \SO(3)$. 
We will use the same arguments used in \cite{PS}, referring the reader to  \cite[Prop.6.1]{PS} for more details.  
 
 \begin{proposition}\label{prop:extension}
Let $g$ be a $G$-invariant Hermitian metric on $Q_3\setminus S^3$ of the form
\[
g=u(r)(dr^2+\eta^2)
+v(r)\left[{\rm{T}}(r)(x_1^2+x_2^2)
+\frac1{{\rm{T}}(r)}(y_1^2+y_2^2)\right]
+w(r)(x_1\odot y_2-x_2\odot y_1). 
\]
Then $g$ extends smoothly across the singular orbit $S^3$ if and only if $u,v,w$ admit smooth extensions to $(-\varepsilon,\varepsilon)$,  
for some $\varepsilon >0$, with
$$
u(r),w(r)\ {\rm{even}},\qquad v(r)\ {\rm{odd}},
$$
and
$$
u(0)=v'(0),\qquad w(0)=0.
$$
The extension is positive definite along the singular orbit $S^3$ if and only if
$$
u(0)=v'(0)>0.
$$

\end{proposition}

\begin{proof} Following \cite[Proposition~6.1]{PS}, a tubular neighbourhood of the singular orbit can be written as
$$
G\times_L V,\qquad L=\SO(3),\qquad V\simeq\mathbb R^3,
$$
where $L$ acts on $V$ by its standard representation. Moreover,
$$
\mathfrak n\coloneqq T_{eL}(G/L)\simeq\mathbb R^3
$$
is also the standard $L$-module and can therefore be identified with $V$. Let $q=[e,0]\in G\times_LV$ and consider the tangent vector $v_3:= \g_r'(0)\in V$ so that $\g_r$ can be represented as $p_r=[e,rv_3]$. Now consider $v_i = X_i\cdot v_3\in V$, $i=1,2$, given by the image of the infinitesimal isotropy action of $X_i$, so that $\{v_1,v_2,v_3\}$ is a basis with corresponding cartesian coordinates $(z_1,z_2,r)$ on $V$. Along the 
normal transversal curve we have
$$
\widehat X_i\big|_{p_r}
=
r\frac{\partial}{\partial z_i}\bigg|_{p_r},
\qquad
\xi\big|_{p_r}
=
\frac{\partial}{\partial r}\bigg|_{p_r}.
$$
Therefore
$$
\left\{
\frac{\partial}{\partial r},
\frac{\partial}{\partial z_1},
\frac{\partial}{\partial z_2},
\widehat A,\widehat Y_1,\widehat Y_2
\right\}
$$
extends to a smooth local frame in a neighbourhood of $q$.
Let
$$
\{dr,dz_1,dz_2,\alpha,\beta_1,\beta_2\}
$$
denote its dual coframe. Along the normal transversal curve, the forms $\alpha,\beta_1,\beta_2$ agree, respectively, with $\eta,y_1,y_2$, while
$$
x_i=\frac{1}{r}\,dz_i,\qquad i=1,2.
$$
Recalling that
$$
\omega=
-u\,dr\wedge\eta
-v\left(x_1\wedge y_1+x_2\wedge y_2\right)
+\T\ w\,x_1\wedge x_2
+\frac{w}{\T}\,y_1\wedge y_2,
$$
we obtain
$$
\begin{aligned}
\omega_{p_r}
={}&-u(r)\,dr\wedge\alpha
+\frac{{\rm{T}}(r)w(r)}{r^2}\,dz_1\wedge dz_2\\
&-\frac{v(r)}{r}
\left(dz_1\wedge\beta_1+dz_2\wedge\beta_2\right)
+\frac{w(r)}{{\rm{T}}(r)}\,\beta_1\wedge\beta_2.
\end{aligned}
$$
We look for conditions equivalent to the extendability of the form $\o$, as we already know that $J$ extends, and the extendability of $g$ follows. 
We restrict $\o$ to the slice $V\smallsetminus\{0\}$ and obtain an $L$-equivariant map 
\beq\label{tilde}\tilde \o:V\smallsetminus\{0\}\to \Lambda^2(V^*+\gn^*) \cong 
\Lambda^2V^*+(V^*\otimes\gn^*)
+ \Lambda^2\gn^*.\eeq
It is clear that $\omega$ extends smoothly if and only if each component of $\tilde\o$ in the above decomposition does extend smoothly across the singular orbit. The component along $\Lambda^2V^*\cong V$ is determined by the restriction to $p_r$, namely by $\frac{{\rm{T}}(r)w(r)}{r^2}\,dz_1\wedge dz_2$. It extends smoothly if and only if the function $\frac{{\rm{T}}(r)w(r)}{r^2}$ extends smoothly as an odd function. This happens precisely when $w$ extends as an even smooth function with $w(0)=0$.  

The component along $\Lambda^2\gn^*\cong \gn$ is determined by the term $\frac{w(r)}{{\rm{T}}(r)}\,\beta_1\wedge\beta_2$ and again the component $\tilde\o^{\Lambda^2\gn^*}$ extends smoothly if and only if the function $\frac{w(r)}{{\rm{T}}(r)}$ extends smoothly as an odd function, leading to the same condition on $w(r)$, namely eveness with $w(0)=0$.

It remains to consider the component of $\widetilde\omega$ in $V^*\otimes\gn^*$. This is determined by the expression along the curve $p_r$
$$\tilde\o^{V^*\otimes \gn*}|_{p_r}= -u(r) dr \wedge\alpha - \frac{v(r)}r (dz_1\wedge \b_1 + dz_2\wedge \b_2).$$
We now define the global $L$-invariant forms 
\[
\begin{split}
\eta &\coloneqq dr\wedge\a + dz_1\wedge\b_1+ dz_2\wedge \b_2, \\
\l &   \coloneqq  rdr+z_1dz_1+z_2dz_2,\quad \mu \coloneqq r\a+z_1\b_1+z_2\b_2.
\end{split}
\]
Then 
$$\l\wedge\mu|_{p_r} = r^2 dr\wedge\a,$$
so that 
 $$\tilde\o^{V^*\otimes \gn*}|_{p_r} = -\frac{v(r)}r \eta|_{p_r} + \frac 1{r^2}\Bigl[\frac{v(r)}r-u(r)\Bigr] (\l\wedge\mu)|_{p_r}.$$
It follows from the same equivariance argument used in \cite[Proposition~6.1]{PS} that the mixed component extends smoothly across the origin if and only if the functions 
$$
\frac{v(r)}r, \qquad
\frac{u(r)-\frac{v(r)}r}{r^2}
$$
extend as smooth even functions of $r$. It is easily seen that this is actually equivalent to saying that 
$$v(r)\ {\rm{odd}},\qquad u(r)\ {\rm{even}},\qquad u(0) = v'(0).$$
The last claim is clear from the expression of the metric.
\end{proof}

In our case, when we look for pluriclosed metrics, the function $w(r)$ has been determined in \eqref{w} and it satifies the extendability conditions, 
as $w(0)=0$ and it is even for every $c\in\mathbb R$. Moreover, condition (1) in Lemma \ref{cond} implies the extendability condition $u(0)=v'(0)$. 
Therefore the extendability conditions boil down to the following
\begin{lemma}\label{lem:extend} 
The $G$-invariant metric $g$ on $Q^3\smallsetminus S^3$ as in \eqref{g} is pluriclosed and extends smoothly across th esingual orbit if and only if the following conditions are satisfied
\begin{itemize}
\item[i)] $u(r),v(r)$ extend as smooth functions on a suitable neighborhood of the origin with $v(r)$ odd and $u(r)=v'(r)$ and $v'(0)>0$;
\item[ii)] $w(r) = c\ (r\coth(2r)-\frac 12),\qquad c\in\mathbb R$;
\item[iii)] $v(r),v'(r) >0$ and $v(r)> |w(r)|$ for every $r>0$.
\end{itemize}
\end{lemma}

\section{Computation of the Bismut-Ricci form}\label{sect.rho}
We denote by $\n^B$ the Bismut connection of the metric $g$ and by $\rho^B$ its Ricci form which is defined, in accordance with the standard formula in the K\"ahler case, as 
$$\rho^B(X,Y) = \frac 12 \sum_{i=1}^6 g(R^B_{XY}Je_i,e_i),$$
where $R^B_{XY}= [\n^B_X,\n^B_Y]-\n^B_{[X,Y]}$. 
We also use the Lee form $\theta$ which is defined as 
$$d(\omega^2) = \theta\wedge \omega^2$$
and satisfies $\theta = Jd^*\omega$, where $J\theta=-\theta\circ J$. We then have the formula (see \cite{IP}, but beware that the conventions regarding $\o$ and $\rho$ differ by a sign)
$$\rho^B = \rho^C + d(J\theta).$$
Moreover we know that we can express 
$$\rho^C = \frac 12 dd^c\log(||\Omega||^2),$$ 
where $\Omega$ is a holomorphic $(3,0)$-form on $Q_3$ (see e.g. \cite{Bes}). Such a form does exist on $Q_3$, which can be realized as the complex homogeneous space $\SO(4,\mathbb C)/\SO(3,\mathbb C)$ with $SO(3,\mathbb C) = \SO(4,\mathbb C)_q$, $q=e_4$. Indeed, see \cite[Lemma 2]{S}, we can write $\so(4,\mathbb C) = \so(3,\mathbb C) + \gq$, where $\gq\cong\mathbb C^3$ is the $\ad(\so(3,\mathbb C))$-invariant complement given by the span of $A,Y_1,Y_2$. As $\Lambda^3\gq^*$ is a trivial module, we have a nonzero $\Omega_o\in \Lambda^3\gq^*$ which extends to a $\SO(4,\mathbb C)$-invariant holomorphic $(3,0)$-form $\Omega$. We can also normalize it so that $\Omega_o(A,Y_1,Y_2)=1$.
Our aim is now to compute the value of $\Omega$ along the curve $\g_r$. We first observe that the curve $\g_r$ can be expressed as the orbit of a one-parameter subgroup in $\SO(4,\mathbb C)$. Indeed, we have 
$$\g_r = a_r\cdot e_4,\quad a_r:= exp(irA).$$
We also consider along $\g_r$ the three $(1,0)$ vectors 
$$\begin{matrix}Z_0 =& \hat A - iJ\hat A &=& \hat A - i \xi\\
Z_1 =& \hat X_1 - iJ\hat X_1 &=& \hat X_1 + i {\rm{T}}(r) \hat Y_1\\
Z_2 =& \hat X_2 - iJ\hat X_2 &=& \hat X_2 + i {\rm{T}}(r) \hat Y_2.
\end{matrix}$$
We aim at computing 
$$\Omega|_{\g_r}(Z_0,Z_1,Z_2) = \Omega|_{a_r q}(Z_0,Z_1,Z_2) = \Omega_q(\Ad(a_r)^{-1}Z_0,\Ad(a_r^{-1})Z_1,\Ad(a_r^{-1})Z_2).$$ 
Moreover at the point $q$, if $\pi:\so(4,\mathbb C)\to \gq$ is the projection, we have 
$$\Omega|_{\g_r}(Z_0,Z_1,Z_2) = \Omega_q(\pi\Ad(a_r)^{-1}Z_0,\pi\Ad(a_r^{-1})Z_1,\pi\Ad(a_r^{-1})Z_2).$$
An easy calculation shows that 
$$\pi\Ad(a_r)^{-1}Z_0 = 2A$$
and 
$$\Ad(a_r^{-1})Z_j = (\ch(r)+{\rm{T}}(r) \sh(r))X_j+i(\sh(r)+\T(r)\ch(r))Y_j,\quad j=1,2$$ 
so that 
$$\pi\Ad(a_r^{-1})Z_j = 2i\ \sh(r) Y_j, \quad j=1,2.$$
This implies that 
$$\Omega|_{\g_r}(Z_0,Z_1,Z_2) = \Omega_q(2A,2i\ \sh(r) Y_1,2i\ \sh(r) Y_2) = - 8\ \sh^2(r).$$ 
Now, the norm of the form $\Omega$ along the curve $\g_r$ is given by 
$$||\Omega|_{\g_r}||_g^2 = \frac{|\Omega(Z_0,Z_1,Z_2)|^2}{\det(g(Z_i,\bar Z_j))}$$
and a straightforward calculation gives 
\beq\label{Norm}
||\Omega|_{g_r}||_g^2 = \frac{2\ \sh(2r)^2}{u(v^2-w^2)}.
\eeq

\begin{lemma}\label{ddc} Let $f\in C^\infty(M)^G$. Then 
$$dd^cf|_{\g_r} = -f''(r)\ dr\wedge \eta - f'(r)(x_1\wedge y_1+ x_2\wedge y_2).$$
\end{lemma}
\begin{proof} Using $\mathcal L_{\hat B}f=0$ for every $B\in\ggg$, we have 
$$dd^cf(\xi,\hat A) = -\xi df(JA) = -f''(r),$$
\begin{align*}dd^cf(X_i,Y_i) &= - d^cf([\hat Y_i,\hat X_i]) = d^cf([\hat X_i,\hat Y_i])\\
{}&= -d^cf(\widehat{[X_i,Y_i]}) = d^cf(\hat A) = -df(J\hat A) \\
{}&= -f'(r).\end{align*}\end{proof} 

\subsection{The Lee form.} 
We compute the invariant Lee form $\theta$, whose restriction to the curve $\g_r$ is a linear combination $a(r) dr + b(r) \eta$ 
by a representation theory argument. 
By definition, 
$$d(\o^2) = \theta\wedge \o^2,$$
hence in our case $2d\o \wedge\o = \theta \wedge \o^2$.
Using \eqref{omega},\eqref{domega} and the fact that $u=v'$ for pluriclosed metrics, 
$$2 d\o \wedge \o = 2w\Bigl( \frac 1\T (\T w)' + \T\left(\frac w\T\right)'\Bigr)\ dr\wedge x_1\wedge x_2 \wedge y_1 \wedge y_2.$$
As $ \frac 1\T (\T w)' + \T(\frac w\T)' = 2w'$ we get 
$$2d\o \wedge \o = 4ww'\ dr\wedge x_1\wedge x_2 \wedge y_1 \wedge y_2.$$
On the other hand 
$$\theta\wedge\o^2 = (a dr+ b \eta)\wedge [2 v^2 x_1\wedge y_1 \wedge x_2 \wedge y_2 + 2 w^2 x_1\wedge x_2 \wedge y_1 \wedge y_2].$$
If we compare this with $d\o\wedge \o$ , we immediately see that $b\equiv 0$. Hence 
$$2ww' = -a(r)[v^2-w^2],$$
hence 
$$\theta = -\frac{2ww'}{v^2-w^2} dr.$$
As $J\theta(X)=-\theta(JX)$, we have 
$$J\theta = \frac{2ww'}{v^2-w^2} \eta.$$
Now we easily verify that 
$$d\eta = x_1\wedge y_1 + x_2 \wedge y_2,$$
so that 
$$d(J\theta) = \Bigl(\frac{2ww'}{v^2-w^2}\Bigr)' dr\wedge \eta + \frac{2ww'}{v^2-w^2} (x_1\wedge y_1 + x_2 \wedge y_2).$$
\subsection{The Bismut Ricci form} We put 
$$P(r) \coloneqq \frac{2ww'}{v^2-w^2}, \quad f(r) \coloneqq \log(||\Omega||^2).$$
then
$$\rho^B = \left(-\frac 12 f''+ P'\right)\ dr\wedge \eta + \left(P- \frac 12 f'\right)\ (x_1\wedge y_1 + x_2 \wedge y_2).$$
Therefore the equation $\rho^B=0$ reduces to 
$$P- \frac 12 f' = 0.$$
This can be written as 
$$\frac{2ww'}{v^2-w^2} - \frac 12\left( 4\coth(2r) - \frac{u'}{u} - \frac{2vv'- 2ww'}{v^2 - w^2}\right)=0.$$
If we now use the fact that $u=v'$, we get
$$\frac{2ww'}{v^2-w^2} - 2\coth(2r) +\frac 12 \frac{v''}{v'} + \frac{vv'- ww'}{v^2 - w^2}=0,$$
hence 
$$- 2\coth(2r) +\frac 12 \frac{v''}{v'} + \frac{vv'+ ww'}{v^2 - w^2}=0,$$
that can be written as 
\beq\label{equation}v'' + 2v'\frac{vv'+ ww'}{v^2 - w^2}- 4v'\coth(2r)=0.\eeq

If $w=0$, i.e., in the K\"ahler case, the equation becomes 
$$\frac{v''}{v'} + 2 \frac{v'}v - 4 \coth(2r) = 0.$$
This equation can be integrated explicitely to yield 
$$v(r) = a (\sh(4r) - 4r)^{1/3}, \quad a\in \mathbb R,$$
and this provides a metric homothetic to the Stenzel metric \cite{S} on $Q_3$.

\smallskip

If $c=1$, the solution is $v(r)=r$ and the corresponding metric is homothetic to the Chamseddine--Volkov/Maldacena-Nu\~nez's metric \cite{CV1,CV2,MN}, 
see Remark \ref{rem:linkphys} for more details.

\smallskip

In what follows,  we aim at proving the existence of a global solution $v\in C^\infty(\mathbb R)$ 
to the equation \eqref{equation} for every $w(r)$ as in \eqref{w}, satisfying the following conditions:
\begin{itemize}
\item[i)] $v$ extends smoothly across $r=0$ to a smooth odd function with $v'(0)=1$;
\item[ii)] $v'(r)>0$ for every $r\in \mathbb R$;
\item[iii)] $v(r)> |w(r)|$ for every $r>0$;
\item[iv)] if $(0,T)$ is the maximal existence interval for the solution $v$ satisfying i) and ii), then 
$$\int_0^T \sqrt{v'(r)}\ dr = +\infty.$$
\end{itemize}

\begin{remark} Note the condition iv) is precisely equivalent to the completeness of the resulting metric $g$ on $Q_3$, as $g(\xi,\xi) = u(r) = v'(r)$ and the completeness condition requires the length of the curve $\g_r$ to be infinite.  

Moreover, the condition $v'(0)=u(0)=1$ normalizes the metric as it fixes the induced homogeneous metric on the singualr orbit $S^3$: indeed in this case the tangent vectors $\hat A|_q,\hat Y_1|_q, \hat Y_2|_q$ have unitary norm. \end{remark}

\section{Existence and properties of the solutions}

\subsection{Existence of a local solution around $r=0$.} 
In this section, we prove the local existence of a solution to \eqref{equation}  
with the required behaviour. 
We shall use the following parameter-dependent version of the
Briot--Bouquet theorem regarding singular first-order partial differential systems.
\begin{theorem}[Briot--Bouquet system with parameters]\label{Liparameter}
Let $F:\C\times\C^q\times\C^m\to \C^q$,
$$F = F(r,U,\lambda),$$
be holomorphic in a neighborhood of $(0,0,0)\in\C\times\C^q\times\C^m$, and assume that
$$
F(0,0,\lambda)=0
$$
for $\lambda$ sufficiently close to $0$.  Let
$$
A=\partial_UF(0,0,0).
$$
If $A$ has no eigenvalue which is a positive integer, then, after possibly shrinking the neighbourhoods of the origin, the
system
$$
r\frac{\partial U}{\partial r}=F(r,U,\lambda)
$$
admits a unique solution $U=U(r,\lambda)$ which is holomorphic in $(r,\lambda)$ and satisfies
$$
U(0,\lambda)=0.
$$
\end{theorem}
This result is an immediate special case of
\cite[Theorem~1.2]{L}, where the parameter $\lambda$ can be seen as the auxiliary variable $x$ in \cite{L}.\par

\smallskip

We set 
\beq\label{wc}w_c(r) \coloneqq  c \left(r\coth(2r)-\frac12\right),\quad r\in\mathbb R\eeq
and we prove the following. 

\begin{proposition}\label{local}
Fix $c\in\mathbb R$. Then the singular ODE 
$$v''+2v'\frac{vv'+w_cw_c'}{v^2-w_c^2}-4v'\coth(2r)=0,\qquad
v(0)=0,\qquad v'(0)=1,$$
admits a unique analytic solution $v_c$ in a neighbourhood of the origin. 
This solution is odd. Moreover, $v_c$ depends analytically on the parameter $c$.
\end{proposition}
\begin{proof} 
Set
$$f(r)\coloneqq \frac{v(r)}r,\qquad p(r) \coloneqq v'(r),$$
and introduce
$$\alpha(r) \coloneqq r\coth(2r),\qquad h_c(r) \coloneqq \frac{w_c(r)}{r} = c\left(\coth(2r)-\frac1{2r}\right).$$
Note that  $\alpha$ is analytic and even, while $h_c$ is analytic and odd.\par
Since
$$p=f+rf',$$
we have
\begin{equation}\label{equa1}
rf'=p-f.
\end{equation}
Moreover,
$$v^2-w_c^2=r^2(f^2-h_c^2)$$
and
$$vv'+w_cw_c'=r\left[fp+h_c(h_c+rh_c')\right].$$
Consequently, multiplying the singular ODE of the statement by $r$, we obtain
\begin{equation}\label{equa2} rp'=4\ p\  \alpha(r)-2p\,\frac{fp+h_c(h_c+rh_c')}{f^2-h_c^2}.
\end{equation}
The initial conditions are
$$f(0)=p(0)=1.$$
Put
$$y=f-1,\qquad z=p-1.$$
Then \eqref{equa1}--\eqref{equa2} take the
Briot--Bouquet form
$$r\frac{d}{dr}\begin{pmatrix}y\\ z\end{pmatrix} =F(r,y,z,c),
\qquad F(0,0,0,c)=\begin{pmatrix}0\\0\end{pmatrix},$$
where 
\[
F(r,y,z,c) \coloneqq \begin{pmatrix} -y+z \\ 4(z+1)\a - 2 (z+1)\frac{(y+1)(z+1)+ h_c(h_c+rh_c')}{(y+1)^2-h_c^2} \end{pmatrix}
\] 
is globally analytic in all its variables in a suitable neighborhood of $(0,0,0,c)$.\par

The linearization at $(r,y,z)=(0,0,0)$ is independent of $c$.  Indeed,
the first equation gives
$$ry'=-y+z,$$
while, since
$$\alpha(0)=\frac12,\qquad h_c(0)=0,$$
the linear part of the second equation is
$$rz'=2y-2z.$$
Hence
$$D_{(y,z)}F(0,0,0,c)= \begin{pmatrix}-1&1\\2&-2\end{pmatrix},$$
whose eigenvalues are $0,-3$ and therefore we can apply Theorem \ref{Liparameter}. 
Hence for every
$c_0\in\mathbb R$, there exist $\varepsilon,\delta>0$ and a unique solution
$$(f(r,c),p(r,c)),\qquad |r|<\delta,\quad |c-c_0|<\varepsilon,$$
which is real analytic in $(r,c)$.  Hence
$$v_c(r)=r f(r,c)$$
depends real analytically on $(r,c)$ near $(0,c_0)$. 
In particular, after possibly decreasing $\delta$, the family depends real analytically on $c$ on the fixed interval $[-\delta,\delta]$.
Therefore, if $c\to c_0$,
$$v_c\longrightarrow v_{c_0}\qquad\text{in }C^\infty([-\delta,\delta]).$$

Finally, the parity of the solution follows from uniqueness. Indeed, $\alpha(r)$ is even and
$$h_c(r)^2,\qquad h_c(r)\bigl(h_c(r)+rh_c'(r)\bigr)$$
are even functions of $r$.  Hence the right-hand side of the Briot--Bouquet system is invariant under the symmetry $r\mapsto-r$.  Therefore
$$(f(-r,c),p(-r,c))$$
is a solution with the same initial data as $(f(r,c),p(r,c))$.
Uniqueness implies that
$$f(-r,c)=f(r,c),\qquad p(-r,c)=p(r,c).$$
Thus $f$ and $p$ are even and
$$v_c(-r)=-v_c(r).$$
\end{proof}

\subsection{Long time existence of the solutions.} In this section we prove the following 
\begin{theorem}\label{thm:completeness}
The solution $v_c$ to the equation \eqref{equation} with initial condition $v_c(0)=0,\ v_c'(0)=1$ and $w_c$ given by \eqref{wc} satisfies the following properties
\begin{itemize}
\item[i)] $v_{-c}=v_{c}$ for every $c$;
\item[ii)] if $c=\pm 1$ the solution $v_1(r)=r$ for all $r\in\mathbb R$;
\item[iii)]  if $|c|<1$, then the existence interval of $v_c$ is the whole $\mathbb R$;
\item[iv)] if $|c|<1$ we have
$$
v_c\sim M_1 e^{\frac 43 r},\quad v_c' \sim M_2 e^{\frac 43 r},\qquad r\gg 0,
$$
so that 
$$\int_0^{+\infty} \sqrt{v'(s)}\ ds =+\infty$$
and the metric $g_c$ defined by $v_c$ is complete.
\item[v)] 

If $|c|>1$, then the maximal positive-definite solution has a
finite endpoint $T>0$, and
\[
\lim_{r\to T^-}v_c(r)=|w_c(T)|,\qquad
\lim_{r\to T^-}v_c'(r)=0.
\]
The corresponding metric on the domain $\{r<T\}$ is incomplete
and does not extend to a BHE metric on all of $Q_3$. 

\end{itemize}
\end{theorem}
\begin{proof} 
As we already recalled, when $c=0$ the solution to \eqref{equation} is given by $v_0 = a (\sh(4r) - 4r)^{1/3}$, for some $a\in \mathbb R$, 
and corresponds to Stenzel's K\"ahler Ricci-flat metric. In such a case, statements iii) and iv) are immediate. 
From now on, we will then assume $c\neq0$. 

\begin{enumerate}[i)]
\item The equation contains only terms involving $w_c^2$ and $w_cw_c'$, hence is dependent only on $c^2$. Both $v_c$ and $v_{-c}$ satisfy the same equation with the same initial conditions and therefore coincide by uniqueness. \par

\item This is a straightforward calculation.\par 

\smallskip

We now state a general fact. Let $[0,T)$ be the maximal interval, with $T\in \mathbb R\cup\{+\infty\}$, 
on which the solution extends smoothly from the origin and
$v_c(r)^2-w_c(r)^2>0$ for $0<r<T$. 
We claim that $v_c'(r)>0$ for every $r\in (0,T)$. We know that there exists an interval $[0,\d]$ ($\d>0$) where $v_c'>0$ by continuity. 
Let $r_o \coloneqq\sup\{t\in [0,T)|\ v_c'(s)>0,\ s\in [0,t]\}$ and suppose $r_o<T$ so that $v_c'(r_o)=0$. Then for all $r\in [0,r_o)$ we have 
$$\frac{v_c''}{v_c'} = -2\frac{v_cv_c'+w_cw_c'}{v_c^2-w_c^2} + 4\coth (2r),$$
and by integration 
$$\log v_c'(r) - \log v_c'(\d) = \int_\d^r  \left(-2\frac{v_cv_c'+w_cw_c'}{v_c^2-w_c^2} + 4\coth (2s)\right) ds.$$
When $r\to r_o$, the right-hand side converges to a finite number, while $\log v_c'\to -\infty$, a contradiction. \par \smallskip

\item Using the equation, we can write the Taylor expansion of $v_c$ at $r=0$ and we obtain 
\beq\label{taylor}v_c'(r) = 1 + \frac 45 (1-c^2) r^2 + o(r^3)\eeq
and therefore when $|c|<1$, the function $v_c'(r)>1$ in a suitable interval $(0,\varepsilon)$. 
We now claim that $v_c'(r)>1$ on the whole interval $(0,T)$. 
If not, let $\bar r$ be the first point where $v_c'(\bar r)=1$. This implies that $v_c''(\bar r)\leq 0$. 
On the other hand, $v_c(r)\geq r$ on $[0,\bar r]$ and therefore we can write $v_c(r)=r+h(r)$ with $h(r)\geq 0$. A calculation shows that 
at the point $\bar r$ we have 
\beq\label{v''}
v_c''= \frac 2{v_c^2-w_c^2}\left[2(v_c^2-w_c^2) \coth(2\bar r) - v_c - w_cw_c'\right],
\eeq
so that if we put 
$$
A_c(r) \coloneqq \left[2(v_c^2-w_c^2)\coth(2 r) - v_c - w_cw_c'\right],
$$
and substitute $v_c(r) = r+ h(r)$, where $h$ is positive, we obtain 
\beq\label{Ac}A_c(r) = (1-c^2)(2\coth(2r)w_1^2+ w_1w_1') + (4r\coth(2r)-1)h + 2\coth(2r)h^2.\eeq

Since $w_1(r)>0$ and
\[
w_1'(r)=\coth(2r)-2r\frac{1}{\sinh^2(2r)}
=\frac{\frac12\sinh(4r)-2r}{\sinh^2(2r)}>0,
\qquad r>0,
\]

and $4r\coth(2r)-1>0$, we see that 
$$
A_c(\bar r) > 0,
$$
and therefore by \eqref{v''}, we see that $v_c''(\bar r)>0$, a contradiction. 

We also note that $w_1'(r)<1$ for $r>0$, as
\[
1-w_1'(r)
=\frac{2r-\frac12(1-e^{-4r})}{\sinh^2(2r)}>0.
\]

Since
\[
w_1(r)-r=\frac{2r}{e^{4r}-1}-\frac12<0,
\qquad r>0,
\]
we have, for $|c|<1$,
\[
v_c(r)>r>|w_c(r)|,\qquad0<r<T.
\]
In particular,
\begin{equation}\label{estimate}
v_c(r)^2-w_c(r)^2>(1-c^2)r^2.
\end{equation}
Equation \eqref{equation} implies
\[
\frac{d}{dr}\log\left(
\frac{v_c'(v_c^2-w_c^2)}{\sinh^2(2r)}\right)
=-\frac{4w_cw_c'}{v_c^2-w_c^2}.
\]
Define
\[
F(r)\coloneqq \frac{v_c'(r)(v_c(r)^2-w_c(r)^2)}{\sinh^2(2r)}.
\]
The initial expansions give $\lim_{r\to0^+}F(r)=\frac14$,
and integration yields
\begin{equation}\label{F}
F(r)=\frac14\exp\left(-\int_0^r
\frac{4w_c(s)w_c'(s)}{v_c(r)^2-w_c(r)^2}\,ds\right).
\end{equation}
In particular, $0<F(r)\leq\frac14$.

Suppose now, by contradiction, that $T<+\infty$. For any fixed $\delta\in(0,T)$,
\[
v_c(r)^2-w_c(r)^2\geq(1-c^2)\delta^2,
\qquad\delta\leq r<T,
\]
and therefore
\[
0<v_c'(r)
=\frac{F(r)\sinh^2(2r)}{v_c(r)^2-w_c(r)^2}
\leq\frac{\sinh^2(2T)}{4(1-c^2)\delta^2}.
\]
Thus $v_c$ and $v_c'$ remain bounded near $T$, while
$v_c(r)^2-w_c(r)^2$ stays bounded away from zero. The differential
equation consequently extends the solution beyond $T$,
a contradiction. Hence $T=+\infty$. The odd extension gives
the solution on all of $\mathbb R$.

\item 
Set
\[
\alpha\coloneqq \frac{4c^2}{1-c^2}.
\]
Since $0<w_1(r)<r$ and $0<w_1'(r)<1$ for $r>0$, we have
$0\leq w_c(r)w_c'(r)\leq c^2r$. Hence, by \eqref{estimate} and \eqref{F},
\[
\frac{F'(r)}{F(r)}
=-\frac{4w_c(r)w_c'(r)}{v_c(r)^2-w_c(r)^2}
\geq-\frac{\alpha}{r}.
\]
Fix $R>0$. Integrating on $[R,r]$ gives a constant $K>0$ such that
\[
F(r)\geq Kr^{-\alpha},\qquad r\geq R.
\]
For sufficiently large $r$, using $ v_c(r)^2-w_c(r)^2\leq v_c(r)^2$ and
$\sinh(2r)\geq\frac13e^{2r}$, we obtain
\[
(v_c^3)'=3v_c^2v_c'
=3F\sinh^2(2r)\frac{v_c^2}{v_c^2-w_c^2}
\geq K_1r^{-\alpha}e^{4r}
\]
for some $K_1>0$. Integrating over $[r-1,r]$ yields
\[
v_c(r)\geq K_2r^{-\frac13\alpha}e^{\frac43r},\qquad r\gg1,
\]
where $K_2>0$. Since $w_c(r)=O(r)$, it follows that
$w_c(r)^2/v_c(r)^2\to0$ and, in particular,
\[
v_c(r)^2-w_c(r)^2\geq\frac12v_c(r)^2,\qquad r\gg1.
\]
Consequently,
\[
\int_R^{+\infty}\frac{w_c(r)w_c'(r)}{v_c(r)^2-w_c(r)^2}\,dr
\leq K_3\int_R^{+\infty}
r^{1+\frac23\alpha}e^{-\frac83 r}\,dr<+\infty,
\]
after increasing $R$ if necessary. The integrand in \eqref{F}  is
$O(r)$ at the origin. Therefore \eqref{F} implies
\[
\lim_{r\to+\infty}F(r)=L_c>0.
\]
As $(v_c^2-w_c^2))/v_c^2\to1$, we now have
\[
v_c^2v_c'
=F\sinh^2(2r)\frac{v_c^2}{v_c^2-w_c^2}
\sim\frac{L_c}{4}e^{4r}.
\]
Integration gives
\[
v_c(r)^3\sim\frac{3L_c}{16}e^{4r},
\]
and hence
\[
v_c(r)\sim M_1e^{\frac43 r},\qquad
v_c'(r)\sim M_2e^{\frac43 r},
\qquad
M_1=\left(\frac{3L_c}{16}\right)^{1/3},\quad
M_2=\frac43M_1.
\]
Note that argument also includes $c=0$, in which case
$\alpha=0$ and $F\equiv\frac14$. Finally, the behaviour of $v_c'$ implies
\[
\int_0^{+\infty}\sqrt{v_c'(r)}\,dr=+\infty,
\]
so the corresponding metric is complete.

\item

We use the same kind of arguments we proved in iii). 
As $|c|>1$, we know that $v_c(r)'<1$ when $r\neq 0$ in a suitable neighborhood of the origin. We claim that $v_c'< 1$ on the maximal existence interval $(0,T)$. Indeed, if $\bar r$ is the first point where $v_c'(\bar r)=1$, we have $v_c''(\bar r)\geq 0$ and we will prove that $v_c''(\bar r)<0$, giving a contradiction. 
On the interval $(0,\bar r)$ we have 
$$w_c(r) < v_c(r) < r,$$ 
and therefore if we set $v_c(r) =r+h(r)$, then $h(r)$ is negative with $|h|< r-w_c(r)$. We now use the same argument as in iii) and consider the quantity $A_c(r)$ as in \eqref{Ac}. We claim that $A_c(\bar r)<0$. Indeed, the first term $(1-c^2) (2\coth(2r)w_1^2+ w_1w_1')|_{r=\bar r}$ is clearly negative, while the second term 
$$B(h) \coloneqq  (4r\coth(2\bar r)-1)h + 2\coth(2\bar r)h^2$$
has roots $h_1=0$ and $h_2=\frac 12 {\rm{T}}(2\bar r) - 2\bar r$. 
We now show that $h(\bar r)=v_c(\bar r)-\bar r$ belongs to the interval $(\frac 12 {\rm{T}}(2\bar r) - 2\bar r,0)$, proving that $B(h(\bar r))<0$ and therefore $A_c(\bar r)<0$, proving our claim. Indeed, $h(\bar r) > w_c(\bar r)-\bar r$ and 
$$w_c(\bar r)-\bar r > h_2= \frac 12 {\rm{T}}(2\bar r) - 2\bar r,$$
as 
$$w_c(\bar r) + \left(\bar r - \frac 12 {\rm{T}}(2\bar r)\right) >0.$$
Therefore, $v_c'< 1$ for every $r\in (0,T)$ and therefore $w_c<v<r$ on $(0,T)$. Suppose now $T=+\infty$. As $\lim_{r\to +\infty} w_c(r)-r =+\infty$ 
in the case $|c|>1$, we see that there esists a point $r_1$ where $v_c(r_1)=w_c(r_1)$ and $T=r_1$, a contradiction. Moreover  
$$v_c' = F\ \frac{\sinh^2(2r)}{v^2-w_c^2}$$
and $F$ is bounded with $0< F(r)\leq \frac 14$ by \eqref{F}, so that 
$$\lim_{r\to T} v_c'(r) = 0.$$
This implies that the integral $\int_o^T\sqrt{v'(s)}\ ds <+\infty$ and the metric defined by this solution on the interval $(0,T)$ is not complete.

\end{enumerate}
\end{proof}

\begin{remark}\label{rem:holiso}
    The metrics $g_c$ and $g_{-c}$ are holomorphically isometric for all $c\neq 0$. Indeed, $v_c=v_{-c}$ by i) of the previous theorem, and 
    the complex linear reflection $R = \diag(-1,1,1,1)\in\mathrm{O}(4)$ satisfies $R^*g_{c}=g_{-c}$, 
    as it reverses the sign of $x_1\odot y_2 - x_2\odot y_1$. 
\end{remark}

Summing up, by Lemma \ref{lem:extend} and the discussion in Section \ref{sect.rho}, 
every smooth $\mathrm{SO}(4)$-invariant BHE metric on $Q_3$, after a positive rescaling, 
is determined by a parameter $c\in\R$ and a smooth solution $v$ of \eqref{equation} with $v(0)=0$ and $v'(0)=1$. 
By Proposition \ref{local} and ordinary ODE uniqueness on the regular part, this solution coincides with $v_c$. 
Since the metric is defined in all of $Q_3$, Theorem \ref{thm:completeness} forces $|c| \leq 1$, 
and Remark \ref{rem:holiso} allows us to restrict to $c\in[0,1]$.

\begin{remark}\label{rem:linkphys}
    As mentioned in the introduction, the family of metrics $g_c$, $c\in(0,1]$ has already appeared in the physical literature, see e.g.~\cite{CNP, MM}. 
    We can relate it to the family considered in \cite{MM} as follows. 
    To avoid conflicting notation, let $\tilde{g}_\gamma$, $\gamma \geq1$, denote the metric $ds_6^2$ and $C(t)$ denote the function $c(t)$ appearing in \cite[(2.2)]{MM}. 
    The variables $r$ and $t$ and the functions $v(r)$ and $C(t)$ are related by 
    \[
    t=2r,\qquad C(t) = \frac{2}{c}v_c\left(\frac{t}{2}\right). 
    \]
    Using the change of invariant coframe
    \[
    \eta=\frac12(\epsilon_3+A_3),\qquad
    x_1=\frac12(\epsilon_1-e_1),\qquad
    y_2=\frac12(\epsilon_1+e_1),\qquad
    x_2=\frac12(\epsilon_2-e_2),\qquad
    y_1=-\frac12(\epsilon_2+e_2),
    \]
    and the identities
    \[
    \tanh(r) = \coth(2r)-\frac{1}{\sinh(2r)},\qquad
    \coth(r) = \coth(2r)+\frac{1}{\sinh(2r)},
    \]
    one obtains
    \[
    \tilde{g}_\gamma =\frac{4}{c}\,g_c.
    \]
    Moreover, since near the singular orbit
    \[
    C(t)=\gamma^2t+O(t^3),\qquad v_c(r)=r+O(r^3),
    \]
    the parameters are related by
    \[
    \gamma^2=\frac{1}{c}.
    \]
    In particular, the critical value $c=1$ corresponds to
    $\gamma=1$, i.e., to the Chamseddine--Volkov/Maldacena--Nu\~nez metric, and   
    \[
    \tilde{g}_{1}=4g_{1}.
    \]
    We recall, however, that the singular ODE governing these metrics and the properties of the corresponding solution were discussed using only numerical analysis tools.
\end{remark}

\section{Scalar curvature}\label{sect:scalar}
We now compute the scalar curvature of the metric $g$ \eqref{g} on the regular part of $Q_3$. We can write  
\[
g = u(r)dr^2 + g_r, 
\]
where 
\[
g_r = u(r)\eta^2 +v(r)\left[{\rm{T}}(r)(x_1^2+x_2^2) +\frac1{{\rm{T}}(r)}(y_1^2+y_2^2)\right] +w(r)(x_1\odot y_2-x_2\odot y_1). 
\]

Following \cite{EW}, we consider the unit normal vector field $N=\frac{1}{\sqrt{u}}\xi$ to the principal orbits $P_r \cong \SO(4)/\SO(2)$ 
along the curve $\gamma_r$. The shape operator $L_r$ of the orbit $P_r$ is given by 
\[
L_r = \frac{1}{2\sqrt{u(r)}} g_r^{-1}{g'_r}.
\] 
Using the Gauss and Codazzi equations and the Riccati equation for $L_r$, one obtains the expression of the Ricci tensor 
of the metric $g$ in terms of the shape operator $L_r$ and the Ricci tensor of $g_r$. 
Tracing it, one ends up with the expression of the scalar curvature of $g$:
\begin{equation}\label{eq:scal}
\mathrm{Scal}_g = \mathrm{Scal}_{g_r} - \frac{2}{\sqrt{u(r)}} \mathrm{tr}(L_r') - \mathrm{tr}(L_r^2) - (\mathrm{tr}(L_r))^2. 
\end{equation}

The scalar curvature $\mathrm{Scal}_{g_r}$ of the principal orbits can be easily computed using \cite[(7.39)]{Bes} 
and working with respect to the basis
$\{E_1,\ldots,E_5\} = \{A,X_1,X_2,Y_1,Y_2\}$ of $\mathfrak{m}$:
\[
\mathrm{Scal}_{g_r} = 	-\frac14 \sum_{1\leq i,j,l,s\leq 5}g_r^{il}g_r^{js}g_r([E_i,E_j]_{\mathfrak{m}},[E_l,E_s]_{\mathfrak{m}} ) 
					-\frac12 \sum_{1\leq i,j\leq 5} g_r^{ij}B(E_i,E_j), 
\]
where $B(X,Y) = 2\mathrm{tr}(XY)$ is the Cartan-Killing form of $\mathfrak{so}(4)$. 
The latter is represented by the matrix $-4\mathrm{Id}$ with respect to the considered basis. 

We let $\Delta \coloneqq v^2-w^2$ and, as before, $\T = {\rm{T}}(r) \coloneqq \tanh(r)$. Using the bracket relations \eqref{bracket}, 
we obtain 
\[
-\frac14 \sum_{1\leq i,j,l,s\leq 5}g_r^{il}g_r^{js}g_r([E_i,E_j]_{\mathfrak{m}},[E_l,E_s]_{\mathfrak{m}} ) = 
-\frac{u}{\Delta^2}(v^2+w^2) -\frac{v^2}{u\Delta}\left(\T^2+\frac{1}{\T^2}\right) + 2 \frac{w^2}{u\Delta},
\]
and
\[
-\frac12 \sum_{1\leq i,j\leq 5} g_r^{ij}B(E_i,E_j) = 4\frac{v}{\Delta}\left(\T+\frac{1}{\T}\right) + \frac{2}{u}, 
\]
so that
\[ 
\mathrm{Scal}_{g_r} = 4\frac{v}{\Delta}\left(\T+\frac{1}{\T}\right) -\frac{v^2}{u\Delta}\left(\T-\frac{1}{\T}\right)^2  
					-\frac{u}{\Delta^2}(v^2+w^2). 
\]

We now compute the remaining contributions in the RHS of \eqref{eq:scal}:
\begin{itemize}
\item  since $\det(g_r) = u\Delta^2$, we obtain
\[
\tr(L_r) 	= \frac{1}{2\sqrt{u}} \tr(g_r^{-1}g_r') = \frac{1}{2\sqrt{u}} \frac{(\det(g_r))'}{\det(g_r)} 
		= \frac{1}{2\sqrt{u}} \frac{(u\Delta^2)'}{u\Delta^2} =  \frac{1}{2\sqrt{u}} \left(\frac{u'}{u}+2\frac{\Delta'}{\Delta}\right). 
\]
We let
\[
K = K(r) \coloneqq \frac{u'}{u}+2\frac{\Delta'}{\Delta},
\]
so that 
\[
(\tr(L_r))^2 = \frac{1}{4u}K^2;
\]

\item we have
\[
-\frac{2}{\sqrt{u}}\tr(L_r') = -\frac{2}{\sqrt{u}}(\tr(L_r))' = -\frac{2}{\sqrt{u}}\left(\frac{1}{2\sqrt{u}} K\right)' 
					= \frac12 \frac{u'}{u^2}K -\frac{1}{u}K' ;
\]

\item the matrix representing $g_r$ with respect to the basis $\{\hat{A}, \hat{X_1}, \hat{Y}_2, \hat{X}_2, \hat{Y}_1\}$ has the 
block-diagonal form 
\[
g_r = 
\begin{pmatrix}
u & 0 & 0 \\
0 & B & O_2 \\
0 & O_2 & C
\end{pmatrix},\qquad \mbox{where}\qquad
B = \begin{pmatrix} v\T & w \\ w & \frac{v}{\T} \end{pmatrix},\qquad
C = \begin{pmatrix} v\T & -w \\ -w & \frac{v}{\T} \end{pmatrix},
\]
so $L_r^2$ is represented by the matrix
\[
L^2 = \frac{1}{4u}(g^{-1}g')^2 =
\frac{1}{4u}\begin{pmatrix}
\left(\frac{u'}{u}\right)^2 &0 &0 \\
0 & (B^{-1}B')^2 & O_2\\
0 & O_2 & (C^{-1}C')^2
\end{pmatrix}. 
\]
Therefore
\[
\begin{split}
\tr(L_r^2) 	&= \frac{1}{4u} \left[ \left(\frac{u'}{u}\right)^2 + \tr((B^{-1}B')^2) + \tr((C^{-1}C')^2) \right]\\
		&= \frac{1}{4u} \left[ \left(\frac{u'}{u}\right)^2 + (\tr(B^{-1}B'))^2 -2\det(B^{-1}B') + (\tr(C^{-1}C'))^2 - 2\det(C^{-1}C') \right]\\
		&= \frac{1}{4u} \left[ \left(\frac{u'}{u}\right)^2 +2\left(\frac{\Delta'}{\Delta}\right)^2 
			-\frac{4}{\Delta}\left((v')^2-(w')^2-v^2\left(\frac{\T'}{\T}\right)^2\right) \right],
\end{split}
\]
where we used
\[
\tr(B^{-1}B') = \frac{\Delta'}{\Delta} = \tr(C^{-1}C'), 
\]
and
\[
\det(B^{-1}B') = \frac{1}{\Delta}\left((v')^2-(w')^2-v^2\left(\frac{\T'}{\T}\right)^2\right) = \det(C^{-1}C'). 
\]
\end{itemize}

Substituting all expressions in \eqref{eq:scal}, we obtain 
\[
\begin{split}
\mathrm{Scal}_g 	&= \mathrm{Scal}_{g_r} - \frac{2}{\sqrt{u}} \mathrm{tr}(L_r') - \mathrm{tr}(L_r^2) - (\mathrm{tr}(L_r))^2 \\
				&= 4\frac{v}{\Delta}\left(\T+\frac{1}{\T}\right) -\frac{v^2}{u\Delta}\left(\T-\frac{1}{\T}\right)^2  
					-\frac{u}{\Delta^2}(v^2+w^2) 
					+ \frac12 \frac{u'}{u^2}K -\frac{1}{u}K' \\
				&\quad -\frac{1}{4u} \left[ \left(\frac{u'}{u}\right)^2 +2\left(\frac{\Delta'}{\Delta}\right)^2 
					-\frac{4}{\Delta}\left((v')^2-(w')^2-v^2\left(\frac{\T'}{\T}\right)^2\right) \right] 
					- \frac{1}{4u}K^2. 
\end{split}
\]
Now, since $\T=\tanh(r)$, we have
\[
\frac{\T'}{\T} = \frac{2}{\sinh(2r)} = \frac{1}{\T}-\T,\qquad \T+\frac{1}{\T} = 2\coth(2r). 
\]
Moreover, 
\[
 \frac12 \frac{u'}{u^2}K -\frac{1}{4u} \left[ \left(\frac{u'}{u}\right)^2 +2\left(\frac{\Delta'}{\Delta}\right)^2\right]- \frac{1}{4u}K^2 = 
 -\frac{3}{2u}\left(\frac{\Delta'}{\Delta}\right)^2.
\]
Therefore, we can rewrite the scalar curvature as follows
\[
\mathrm{Scal_g} = 8\frac{v}{\Delta}\coth(2r) -\frac{8v^2}{u\Delta\sinh^2(2r)}  
					-\frac{u}{\Delta^2}(v^2+w^2) 
					 -\frac{1}{u}K' 
					+\frac{(v')^2-(w')^2}{u\Delta} 		
					 -\frac{3}{2u}\left(\frac{\Delta'}{\Delta}\right)^2.
\]

\medskip

Let us now assume that the metric $g$ is pluriclosed. 
Lemma \ref{cond} and the discussion following it ensure that  $u=v'$ and $w = c\left(r\coth(2r)-\frac12\right)$, with $c\in\R$. 
The scalar curvature becomes then
\[
\mathrm{Scal_g} = 8\frac{v}{\Delta}\coth(2r) -\frac{8v^2}{v'\Delta\sinh^2(2r)}  
					-\frac{v'}{\Delta^2}(v^2+w^2) 
					 -\frac{1}{v'}\left(\frac{v''}{v'}+2\frac{\Delta'}{\Delta}\right)' 
					+\frac{(v')^2-(w')^2}{v'\Delta} 		
					 -\frac{3}{2v'}\left(\frac{\Delta'}{\Delta}\right)^2.
\]

Finally, we consider the condition $\rho^B=0$, which is equivalent to \eqref{equation}. This can be rewritten as
\[
\frac{v''}{v'} = 4\coth(2r) - 2\frac{vv'+ ww'}{\Delta}. 
\]
Using this identity, we have
\[
\begin{split}
K' &= \left(\frac{v''}{v'}+2\frac{\Delta'}{\Delta}\right)' = \left(4\coth(2r)+\frac{2vv'-6ww'}{\Delta}\right)' \\
    &= -\frac{8}{\sinh^2(2r)} + \frac{2(v')^2 +2vv''-6(w')^2-6ww''}{\Delta} - \frac{\Delta'}{\Delta^2}(2vv'-6ww')\\
    &= -\frac{8}{\sinh^2(2r)} + \frac{2(v')^2 +2vv''-6(w')^2-6ww''}{\Delta}- \frac{\Delta'}{\Delta^2}(\Delta'-4ww')\\
    &= -\frac{8}{\sinh^2(2r)} + \frac{2(v')^2 +2vv''-6(w')^2-6ww''}{\Delta}- \left(\frac{\Delta'}{\Delta}\right)^2  + \frac{4\Delta'ww'}{\Delta^2}. 
\end{split}
\]
Now, using again the above expression of $\frac{v''}{v'}$ and the identity $w''=\frac{8}{\sinh^2(2r)}w$, a tedious but straightforward computation 
leads to the expression of $\mathrm{Scal}_g$ given in the next result. 
\begin{proposition}\label{prop:scal}
    The scalar curvature of the $\mathrm{SO}(4)$-invariant Bismut Hermitian-Einstein metric $g_c$ on $Q_3\smallsetminus S^3$ along the curve $\gamma_r$ is given by
    \[
    \mathrm{Scal}_{g_c} = \frac{(5v_c^2+w_c^2)(w_c')^2}{v_c'\Delta_c^2} 
                    +\frac{40w_c^2}{v_c'\Delta_c \sinh^2(2r)},
    \]
    where $w_c=c\left(r\coth(2r)-\tfrac12\right)$ and $\Delta_c = v_c^2-w_c^2$. 
\end{proposition}

\begin{remark}
Recall that $g_0$ is Stenzel's K\"ahler Ricci-flat metric. In such a case $w_0=0$ and thus $\mathrm{Scal}_{g_0}=0$. 
\end{remark}

We now compute the scalar curvature of the family $g_c$ on the singular orbit and, when $|c|\leq1$, its behavior at infinity. 

By \eqref{taylor} the Taylor expansion of $v_c$ at $r=0$ is given by
\[
v_c(r) =  r + \frac{4}{15} (1-c^2) r^3 + o(r^4).
\]
Moreover, the Taylor expansion of $w_c(r)$ at $r=0$ is
\[
w_c(r) = \frac23c r^2  - \frac{8}{45}cr^4 + o(r^4).
\]
Therefore, we have
\[
\lim_{r\to0^+}\mathrm{Scal}_{g_c} = \lim_{r\to0^+}\frac{(5v_c^2+w_c^2)(w_c')^2}{v_c'\Delta_c^2} 
                                    +\lim_{r\to0^+}\frac{40w_c^2}{v_c'\Delta_c \sinh^2(2r)} = \frac{80}{9}c^2+\frac{40}{9}c^2 = \frac{40}{3}c^2, 
\]
and we obtain the following. 
\begin{proposition}
    The metric $g_c$ has positive scalar curvature for all $c\in(0,1]$. 
\end{proposition}
\begin{proof}
    Since $c\neq 0$, the scalar curvature on the singular orbit $S^3$ is positive and equal to $\tfrac{40}{3}c^2$. 
    The expression of $\mathrm{Scal}_{g_c}$ in Proposition \ref{prop:scal} and the results of Theorem \ref{thm:completeness} show that 
    the scalar curvature of $g_c$ is positive also on the regular part of $Q_3$ when $c\in(0,1]$. 
\end{proof}

Let us now discuss the behaviour for $r\to+\infty$. By Theorem \ref{thm:completeness}, iv), we have for $|c|<1$
\[
\frac{(5v_c^2+w_c^2)(w_c')^2}{v_c'\Delta_c^2} \sim  e^{-4r}, \qquad 
\frac{40w_c^2}{v_c'\Delta_c \sinh^2(2r)} \sim r^2 e^{-8r},
\]
so that 
\[
\lim_{r\to +\infty}\mathrm{Scal}_{g_c} = 0. 
\]

On the other hand, when $c=1$ we have the Chamseddine--Volkov/Maldacena-Nu\~nez metric $g_1$ with $v_1(r)=r$ and thus
\[
\lim_{r\to +\infty}\mathrm{Scal}_{g_1} = \lim_{r\to +\infty}\frac{(5r^2+w_1^2)(w_1')^2+40w_1^2(r^2-w_1^2)\operatorname{csch}^2(2r)}{(r^2-w_1^2)^2} = 6.
\]

These results allow us to prove the following. 
 
\begin{proposition}\label{prop:inhomogeneous}
    The metrics $g_c$ with $c\in [0,1]$ are non-homogeneous and pairwise non-isometric. Moreover, each $g_c$ is not Bismut flat.  
\end{proposition}
\begin{proof} 
Recall from Remark \ref{rem:holiso} that the metrics $g_c$ and $g_{-c}$ are holomorphically isometric for every $c\neq0$, so we can 
focus on $c\in[0,1]$.
The metrics $g_c$, $0<c\leq 1$ have non-constant scalar curvature, hence they are not homogeneous. 
The Stenzel's metric, when $c=0$, is not homogeneous, too, as otherwise it would be flat by Alekseevsky-Kimelfeld theorem \cite{AK}.\par 
The group $G=\SO(4)$ is an isometry group for all metrics $g_c$. If $\phi:(Q_3,g_c)\to (Q_3,g_{c'})$ is an isometry, then $\phi_*G$ is an isometry group of $(Q_3,g_{c'})$ that has the same principal orbits as $G$, because $g_{c'}$ is not homogeneous. Therefore  the singular orbit $S^3$ is $\phi$-stable. The expression of the scalar curvature on the orbit $S^3$ shows that $c=c'$.

As for the last claim, we use \cite[Theorem 5]{Z}, which ensures that completeness and Bismut flatness 
would imply the manifold to be a Samelson space, hence homogenous, a contradiction. \end{proof}

\section{The Bismut Holonomy} 
In this section we compute the holonomy of the Bismut connection $\nabla^B = \nabla^{LC} - \frac12 g^{-1}d^c\omega$ of the family 
of Bismut Hermitian-Einstein 
metrics $g_c$, for $0<c\leq1$, where $\nabla^{LC}$ denotes the Levi Civita connection of $g_c$. 
It is known that $\mathrm{Hol}(\nabla^B)=\mathrm{SU}(3)$ for the Stenzel metric $g_0$.

We fix a metric $g_c$ and we consider the singular orbit $S^3$, a point $p\in S^3$ and the Bismut holonomy algebra $\gh_c$ at $p$. 
We know that $\gh_c\subseteq \su(3)$ as $g_c$ is Bismut-Ricci flat and moreover $\gh_c\neq \{0\}$ because $g_c$ is not Bismut flat. 
By Ambrose-Singer theorem, we know that the isotropy algebra of the full isometry group normalizes $\gh_c$. 
In our case we consider the subalgebra $\so(3)\subset \su(3)$ given by the isotropy $\frak g_p$. We also note that $(\su(3),\so(3))$ is a symmetric pair, 
so that we can write 
$$\su(3) = \so(3) + \gq,\quad [\gq,\gq] = \so(3),\qquad \gq\cong S^2_o(\mathbb R^3).$$
As an $\so(3)$-module, the algebra $\su(3)$ has only two irreducible summands which are non equivalent, 
hence any non trivial invariant subspace is either $\so(3)$ or $\gq$. But $\gq$ is not a subalgebra, 
hence $\gh_c = \so(3)$ or $\gh_c = \su(3)$.

We claim that $\gh_c = \su(3)$ for all $c\in[0,1]$. To prove this, we divide the discussion into two parts: $c\in[0,1)$ and $c=1$.

\subsection{The Bismut holonomy for $c\in[0,1)$}
This is the simplest case and we prove the claim in the next result. 
\begin{proposition} 
If $c\in [0,1)$, the metrics $g_c$ have full Bismut holonomy, namely $\mathrm{Hol}(\nabla^B)=\SU(3)$.
\end{proposition}
\begin{proof} 
Suppose the holonomy algebra $\gh_c$ reduces to $\so(3)$. 
It then follows that all curvature endomorphisms $R^B_{z_1,z_2}|_p$ preserve the tangent space $T_pS^3$, 
where $R^B$ denotes the Bismut curvature. 
This implies, in particular, that 
$$0=g_c(R^B_{\xi,\hat A}\hat A,\xi)|_p = \lim_{r\to 0^+}g_c(R^B_{\xi,\hat A}\hat A,\xi)|_{\g_r}.$$
On the other hand, we have for $r>0$
\begin{align}\label{curv}
g_c(R^B_{\xi,\hat A}\hat A,\xi) &= g_c(\n^B_\xi\n^B_{\hat A}\hat A,\xi) -g_c(\n^B_{\hat A}\n^B_\xi\hat A,\xi)\notag\\ 
{}&= \frac d{dr}g_c(\n^B_{\hat A}\hat A,\xi) - g_c(\n^B_{\hat A}\hat A,\n^B_\xi\xi) 
{}- Ag_c(\n^B_\xi\hat A,\xi) + g_c(\n_\xi\hat A,\n^B_ {\hat A}\xi).
\end{align}
As $d^c\omega(\hat A,\xi,\cdot)=0$, we have 
$$\n^B_{\hat A}\hat A = \n^{LC}_{\hat A}\hat A, \qquad \n^B_{\hat A}\xi = \n^B_{\xi}\hat A = \n^{LC}_{\hat A}\xi,\qquad \n^B_{\xi}\xi = \n^{LC}_{\xi}\xi.   $$
A simple application of Koszul formula shows that 
$$\n^B_{\hat A}\hat A = -\frac{u'}{2u}\xi,\qquad \n^B_{\xi}\xi = \frac{u'}{2u}\xi,\qquad \n^B_{\hat A}\xi = \frac {u'}{2u}\hat A.$$
It then follows that \eqref{curv} can be written as 
$$g_c(R^B_{\xi,\hat A}\hat A,\xi) = \frac12\Bigl(-u'' + \frac{u'^2}{u}\Bigr). $$
As $u(r) = 1+ \frac 45(1-c^2)r^2 + o(r^3)$, we see that 
$$\lim_{r\to 0}\frac12\Bigl(-u'' + \frac{u'^2}{u}\Bigr) = -\frac 45 (1-c^2),$$
and therefore 
$$g_c(R^B_{\xi,\hat A}\hat A,\xi)|_p \neq 0,\ c\in [0,1),$$
a contradiction. This implies that the holonomy is not $\so(3)$, hence it is full $\su(3)$. 
Since $Q_3$ is simply connected, we conclude that $\mathrm{Hol}(\nabla^B)=\SU(3)$. 
\end{proof}

\subsection{The Bismut holonomy for $c=1$}
This case has to be treated differently, 
as it can be seen that at $p$ the curvature operator satisfies $R^B_p(\Lambda^2(T_pQ_3))\subseteq \so(3)$ 
and therefore the argument used in the previous proof does not apply. 

We assume that the Bismut holonomy is contained in the standard
$\SO(3)\subset \SU(3)$, and derive a contradiction.  Such a reduction is
equivalent to the existence of a $\nabla^B$-parallel real Lagrangian
rank-three subbundle
$$\mathcal E\subset TQ_3,\qquad
TQ_3=\mathcal E\oplus J\mathcal E$$
and we may assume that, at a point $p$ of the singular orbit,
\[
 \mathcal E_o=T_oS^3 .
\]

\begin{lemma} The distribution $\mathcal E$ is $\SO(4)$-invariant.\end{lemma}
\begin{proof} Note that the set of all Lagrangian $\SO(3)$-invariant subspaces of $T_pQ_3$ is in bijection with the real forms of $T^{(1,0)}_pQ_3$, hence with $S^2(T^{(1,0)}_pQ_3)^{\SO(3)}\cong \mathbb C$. Therefore if $h\in \SO(4)$, then $(h_*\mathcal E)_p$ can be identified with $\mathcal E_o$ modulo some complex number. This means that we have a homomorphism $\SO(4)\to U(1)$, which is then necessarily trivial as $\SO(4)$ is semisimple. It follows that $h_*\mathcal E =\mathcal E $.\end{proof}

We now determine the restriction of $\mathcal E$ along the normal
geodesic.
\begin{lemma} We have $\hat A_{\g_r}\in \mathcal E_{\g_r}$
\end{lemma}
\begin{proof} It is enough to prove that 
$\nabla^B_\xi \hat A=0$. We recall that $d^c\omega(\xi,\hat A,\cdot)=0$ and therefore $\nabla^B_\xi \hat A = \nabla^{LC}_\xi \hat A$. As  $g(\hat A,\hat A)=u=1$ and all scalar products of $\hat A$ with $\hat X_i,\hat Y_i,\xi$ vanish, we get our claim.\end{proof}

The principal isotropy group $H\simeq \SO(2)$ fixes
$A,\xi$ and acts by the standard two-dimensional representation on $\gp_1,\gp_2$. Near $r=0$ with $r>0$, the  two-dimensional space 
$\mathcal E\cap(\gp_1+\gp_2)$ is transverse to
$\gp_1$, and hence is the graph of an $H$-equivariant map $
\mathcal \gp_2\longrightarrow\gp_1$.
Since $\End(\mathbb R^2)^{\SO(2)}=\Span\{I,K\}$ where 
$K=\left(\begin{smallmatrix}0&-1\\1&0\end{smallmatrix}\right)$,
there exist functions $a(r),b(r)$ such that
\[
\begin{split}
 U_1&=Y_1+aX_1+bX_2,\\
 U_2&=Y_2-bX_1+aX_2,
\end{split}
\]
and
\beq\label{distr}
\mathcal E_{\gamma_r}=\Span\{\hat A,\hat U_1,\hat U_2\}.
\eeq

Since $\mathcal E$ is Lagrangian, we must have
$\omega(\hat U_1,\hat U_2)=0$ and we obtain
\begin{equation}\label{eq:Lagrangian-ab}
 wT(a^2+b^2)-2vb+\frac{w}{T}=0.
\end{equation}
Set
$$
\Delta \coloneqq v^2-w^2,\qquad s \coloneqq \sqrt \Delta.
$$
Equation \eqref{eq:Lagrangian-ab} can be rewritten as
$$
a^2+
\left(b-\frac{v}{wT}\right)^2
=
\frac{s^2}{w^2T^2}
$$
and therefore we can write 
\[
 a=\frac{s}{wT}\sin\vartheta,
 \qquad
 b=\frac{v-s\cos\vartheta}{wT}.
\]
for some angle $\vartheta(r)$ for $r>0$ sufficiently close to $r=0$.

We next impose
$$\nabla^B_\xi\mathcal E\subset\mathcal E,$$
and we prove the following.
\begin{lemma} We have 
\begin{equation}\label{Radial}
\begin{split}
 \nabla^B_\xi \hat X_1
 &=(\sigma+\rho)\hat X_1+T\kappa\,\hat Y_2,\\
 \nabla^B_\xi \hat X_2
 &=(\sigma+\rho)\hat X_2-T\kappa\,\hat Y_1,\\
 \nabla^B_\xi \hat Y_1
 &=(\sigma-\rho)\hat Y_1-\frac{\kappa}{T}\hat X_2,\\
 \nabla^B_\xi \hat Y_2
 &=(\sigma-\rho)\hat Y_2+\frac{\kappa}{T}\hat X_1.
\end{split}
\end{equation}
where 
\[
 \rho=\frac{T'}{2T},
 \qquad
 \sigma=\frac{vv'-ww'}{2\Delta},
 \qquad
 \kappa=\frac{vw'-v'w}{2\Delta}.
\]
\end{lemma}
\begin{proof} We give the sketch of proof of the first formula in
\eqref{Radial}, the remaining three formulas being obtained in the same way. We first recall that 
$$d^c\omega(\xi,\hat X_1,\hat X_2)= d^c\omega(\xi,\hat X_1,\hat Y_1)=0,\quad d^c\omega(\xi,\hat X_1,\hat Y_2)= \left(T-\frac 1T\right)w = -\frac{T'}T w.$$
Using then Koszul's formula, we obtain
\begin{align*} 2g(\nabla^B_\xi\hat X_1,\hat X_1)&=(vT)'\\ 
2g(\nabla^B_\xi \hat X_1,\hat Y_2)&=
2g(\nabla^{LC}_\xi \hat X_1,\hat Y_2)-d^c\omega(\xi,\hat X_1,\hat Y_2)=
w'+\frac{T'}T w\\
g(\nabla^B_\xi\hat X_1,\hat X_2) &= g(\nabla^B_\xi\hat X_1,\hat Y_1) =0
\end{align*} 
and therefore our claim follows using the inverse of the $2\times2$ metric block $\left(\begin{smallmatrix}vT&w\\w&v/T\end{smallmatrix}\right)$ relative to the subspace generated by $X_1,Y_2$.
\end{proof}
Therefore we obtain 
\begin{align*}\nabla^B_\xi \hat U_1
&=\left(a'+a(\sigma+\rho)\right)\hat X_1+\left(b'+b(\sigma+\rho)-\frac{\kappa}{T}\right)\hat X_2\\
&+\left(\sigma-\rho-bT\kappa\right)Y_1+aT\kappa\,\hat Y_2.
\end{align*}
The condition that this vector belong to $\Span\{U_1,U_2\}$ means that
$\nabla^B_\xi \hat U_1=\lambda \hat U_1+\mu \hat U_2$ and a comparison of the $Y_1,Y_2$-components gives
$$
\lambda=\sigma-\rho-bT\kappa,
\qquad
\mu=aT\kappa.
$$
A comparison of the $X_1,X_2$-components then yields
\begin{equation}\label{RadialEq}
\begin{split}
 a'&=-2\rho\,a-2T\kappa\,ab,\\
 b'&=-2\rho\,b+\frac{\kappa}{T}
     +T\kappa(a^2-b^2).
\end{split}
\end{equation}
The equations obtained from $U_2$ are the same, by $\SO(2)$-equivariance.

Now we note that $\vartheta\equiv 0$ gives a smooth solution. Indeed, if we set $\vartheta\equiv 0$, 
then $a\equiv 0$ and $b=\frac{v-s}{wT} =\frac{w}{T(v+s)}$. We have that for $c=1$,
$$v=r,\qquad w=\frac23r^2-\frac8{45}r^4+O(r^6),
\qquad s=\sqrt{r^2-w^2} =r-\frac29r^3+O(r^5),$$
and hence
$$b(r)=\frac13+\frac{8}{135}r^2+O(r^4)$$
showing that $b$ is smooth in a neighborhood of the origin. Moreover a simple computation shows that $b$ satisfies also the second equation in \eqref{RadialEq}. Therefore 
the distribution determined by \eqref{distr} extends smoothly to
the singular orbit and its limit there is precisely $T_oS^3$.
Since it satisfies the radial parallel transport equations and parallel
transport from $T_oS^3$ is unique, it follows that the 
parallel Lagrangian distribution $\mathcal E$  must be expressed along the curve $\g_r$ as 
\beq\label{Eexplicit}
 \mathcal E_{\gamma(r)} = {\operatorname{Span}}\left\{\hat A,\,\hat Y_1+b\hat X_2,\,
 \hat Y_2-b\hat X_1 \right\},\qquad b=\frac{v-s}{wT}.\eeq
\medskip
\noindent We now show that the $\SO(4)$-homogeneous extension of
\eqref{Eexplicit} cannot be $\nabla^B$-parallel.

Along the curve $\gamma_r$ we consider 
$$U=(\hat Y_1+b(r) \hat X_2)|_{\g_r},\qquad h(r):=g(U,U).$$
Since
$$g(\hat Y_1,\hat Y_1)=\frac vT,\qquad g(\hat X_2,\hat X_2)=vT,\qquad
g(\hat Y_1,\hat X_2)=-w,$$
we have
\begin{equation}\label{eq:h-def}
 h=
 \frac vT+b^2vT-2bw.
\end{equation}
We fix $r_o>0$ and consider a local extension $\tilde U$ of $U_{\gamma_{r_o}}$ as a section of $\mathcal E$. As $\mathcal E$ is parallel, we must have $\n^B_{\tilde U}\tilde U\in \Gamma(\mathcal E)$. In particular, as $\xi=J\hat A$ along $\g_r$, we have 
$$g(\n^B_{\tilde U}\tilde U,\xi)|_{\gamma_r}=0.$$
We now note that 
\beq\label{par}
g(\n^B_{\tilde U}\tilde U,\xi)|_{\gamma_r}=g(\n^{LC}_{\tilde U}\tilde U,\xi)|_{\gamma_r}=0
\eeq
as $d^c\omega(\tilde U,\tilde U,\xi)=0$. As the right-hand side in \eqref{par} is tensorial in $\tilde U$, we can consider the vector field on $Q_3\setminus S^3$
$$V:= \hat Y_1 + b(r)\hat X_2$$
and compute 
$$F(r):= g(\n^{LC}_{\tilde U}\tilde U,\xi)|_{\gamma_r} = g(\n^{LC}_{V}V,\xi)|_{\gamma_r} = 0.$$

We compute the right-hand side directly from the Koszul formula and we obtain
\begin{align*}
2g(\nabla^{LC}_VV,\xi)&=-\xi\,g(V,V)-2g([V,\xi],V)\\
&=-h'+2b'g(\hat X_2,V).
\end{align*}
Now $g(\hat X_2,V)=bvT-w$, 
while differentiating \eqref{eq:h-def} gives
$$
h'
=
\left(\frac vT\right)'
+b^2(vT)'
-2bw'
+2b'(bvT-w).
$$
The terms containing $b'$ cancel, and hence
\[
2F(r)= 2g(\nabla^{LC}_VV,\xi)= - \left[ \left(\frac vT\right)'+b^2(vT)'-2bw'\right].
\]
Recalling that 
$$v=r,\quad b= \frac{v-s}{wT},\quad s=\sqrt{v^2-w^2}$$
a straightforward computation shows that 
$$F(r)= -\frac s{T(r+s)}\Bigl[s'-r\frac{T'}T\Bigr].$$
As $s=r-\frac 29 r^3 + o(r^4)$ we see that 
$$F(r) = -\frac{32}{405}r^3+o(r^4).$$
This proves that $F(r)$ does not vanish in a suitable neighborhood of $r=0$, giving a contradiction.

\bigskip

\noindent{\bf Acknowledgements.} 
In investigating whether the family of metrics studied here had previously appeared in the literature, we consulted ChatGPT, 
which drew our attention to the Chamseddine--Volkov/Maldacena--Nu\~nez metric \cite{CV1,CV2,MN} and related work by 
Butti--Grana--Minasian--Petrini--Zaﬀaroni \cite{BGMPZ},
Casero--Nu\~nez--Paredes~\cite{CNP} and Martelli--Maldacena \cite{MM}.
The references and their relation to our construction were subsequently verified by the authors. \par
After the completion of this work, we became aware of related work by Mario Garc\'ia-Fern\'andez and Lorenzo Foscolo \cite{FG},  which contains results that, in particular, recover our main theorem using a different approach. 
We thank Mario Garc\'ia-Fern\'andez and Lorenzo Foscolo for informing us about their work.

The authors warmly thank Jeffrey Streets for very useful and insightful remarks. The authors acknowledge partial support by GNSAGA (INdAM, Italy).

\end{document}